\documentclass[10pt]{amsart}
\usepackage{amsfonts}
\usepackage{amssymb}
\usepackage{amscd}
\newtheorem{theorem}{Theorem}[section]
\newtheorem{corollary}[theorem]{Corollary}

\newtheorem{lemma}[theorem]{Lemma}
\newtheorem{proposition}[theorem]{Proposition}

\newtheorem{remark}[theorem]{Remark}

\newtheorem*{Theorem A}{Theorem A}
\newtheorem*{Theorem A'}{Theorem A'}
\newtheorem*{Corollary A1}{Corollary A1}
\newtheorem*{Theorem B}{Theorem B}
\newtheorem*{Corollary B1}{Corollary B1}
\newtheorem*{theorem*}{Theorem}
\newtheorem*{Theorem C}{Theorem C}

\theoremstyle{definition}
\theoremstyle{remark}
\numberwithin{equation}{section}

\newcommand{\rmi}{{\rm (i)\,\,}}
\newcommand{\rmii}{{\rm (ii)\,\,}}

\newcommand{\vp}{\varphi}

\newcommand{\diver}{{\mathop{\mathrm div}\,}}
\renewcommand{\div}[1]{{\mathop{\mathrm div}}\bigl(#1\bigr)}

\newcommand{\vol } {\mathrm{vol}\,}
\renewcommand{\a}{\alpha}

\newcommand{\e} {\epsilon}

\newcommand{\ty} {\infty}

 \renewcommand{\O}{\Omega}
\newcommand{\bdr}{\partial}
\newcommand{\R}{\mathbb{R}}

\newcommand{\N} {\mathbb N}
\newcommand{\ds}{\displaystyle}

\begin{document}

\title[Logistic-type equations on manifolds]
{Existence and non-existence results for a logistic-type equation on manifolds}

\author{Stefano Pigola}
\address{Dipartimento di Fisica e Matematica\\
Universit\`a dell'Insubria - Como\\
via Valleggio 11\\
I-22100 Como, ITALY}
\email{stefano.pigola@uninsubria.it}

\author{Marco Rigoli}
\address{Dipartimento di Matematica\\
Universit\`a di Milano\\
via Saldini 50\\
I-20133 Milano, ITALY}
\email{rigoli@mat.unimi.it}

\author{Alberto G. Setti}
\address{Dipartimento di Fisica e Matematica
\\
Universit\`a dell'Insubria - Como\\
via Valleggio 11\\
I-22100 Como, ITALY}
\email{alberto.setti@uninsubria.it}

\dedicatory{Dedicated to the memory of Franca Burrone
Rigoli}



\begin{abstract}
We study the steady state solutions of a generalized logistic type equation on a complete Riemannian manifold.
We provide sufficient conditions for existence, respectively non-existence of  positive solutions, which depend
on the relative size of  the coefficients and their mutual interaction with the geometry of the manifold,
which is mostly taken into account by means of conditions on the volume growth of geodesic balls.
\end{abstract}

\maketitle
\section*{0. Introduction}

The aim of this paper is to study the problem of existence,
non-existence and uniqueness of  solutions of the
equation
\begin{equation}
\label{0.3}
\Delta u + a(x) u - b(x)u^\sigma = 0\quad \text{ on } \, M,
\end{equation}
on a complete, connected Riemannian manifold
$(M, \langle \,,\, \rangle)$. Here $\sigma>1$, the coefficient  $b(x)$  is assumed to
be non-negative while $a(x)$  is not assumed to be of constant sign.

Equations of the form (\ref{0.3}) arise in Riemannian geometry, as the equation for the
change of the scalar curvature under a conformal change of the metric (see, e.g.,
\cite{K}) and  in mathematical biology, where they describe the steady
state solutions of the  logistic equation with diffusion
\begin{equation}
\label{0.1}
\frac{\bdr}{\bdr t} u= \Delta u + a(x) u - b(x) u^\sigma,
\end{equation}
(see, e.g., \cite{AB}, \cite{AW}, \cite{DM1}, \cite{DM2}). In the
latter context $u$  represents the density of a population, and it
is therefore assumed to be nonnegative, the non-linear term $-
b(x) u^\sigma$ accounts for the fact that the population is self
limiting, and the function $a(x)$ represents the birth rate of the
population, with no self limitation.

In  Euclidean setting,
G.A.~Afrouzi and K.J~Brown,
\cite{AB},  have studied the following special case of (\ref{0.3})
\begin{equation}
\label{0.2}
\Delta u + \lambda [g(x) u - u^2] = 0\quad \text{ on } \, \R^m,
\end{equation}
where the positive parameter $\lambda>0$ is the inverse of  the
diffusion rate, and $g(x)$ is a changing sign coefficient which
again represents  the birth rate of the population. Their
results describe the interplay between diffusion, and birth rate,
and show that if diffusion is sufficiently small, solutions may exist even if  $a$ is
predominantly negative, while if diffusion is large, then solutions
exist only if the birth rate is sufficiently large.

In fact, the mutual interactions between diffusion and birth rate
is often taken into account by the principal eigenvalue
$\lambda_*$ of the linear part of equation (\ref{0.2}) (see
Section~1 below for the relevant definitions). This is
exemplified, e.g., in  \cite{AB} Theorem~2.2, where it is proved
that if $g$ is positive somewhere, so that  $\lambda_*\geq 0,$
and if $\lambda>\lambda_*$, then equation (\ref{0.2}) has a positive
solution.

Further, (see Section 3 therein) under the additional assumption
that  $g(x)$ is strictly negative in the complement of a ball,
(\ref{0.2}) has  (exactly one) positive solution if $\lambda
>\lambda_*$, and no positive solution   if $\lambda \leq \lambda_*$.

More recently, existence, non-existence, and uniqueness results have been
obtained by Y.Du and L.Ma,  \cite{DM2}, who study the equation
\begin{equation}
\label{0.2bis}
\Delta u + \lambda g(x) u - b(x)u^\sigma = 0\quad \text{ on } \, \R^m
\end{equation}
where $\sigma>1$ and $b(x)$ is a non-negative coefficient. They show that if
$g$ is positive somewhere and its positive part satisfies suitable
conditions which in particular imply  that $\lambda_*>0$, and if $b$ is strictly
positive outside a bounded connected open set $S_o$ with principal eigenvalue
$\lambda_1(S_o)$ ($=+\infty$ if $S_o$ is empty), then (\ref{0.2bis}) has a (unique)
solution in the homogeneous Sobolev space $H_h^1(\R^m)$ if
$\lambda_*< \lambda<
\lambda_1(S_o)$, no solution in $H_h^1(\R^m)$ if $0<\lambda\leq
\lambda_*$ and no solutions at all if $\lambda_1(S_o)<+\infty$ and
$\lambda\geq \lambda_1(S_o).$

Our results provide new insight on the interplay between diffusion
and growth rate, that is, in our  notation, between the relative size of
the variable coefficients $a(x)$ and $b(x).$ Indeed, we show that
if $a(x)$ is sufficiently large  in a suitable ball,
while outside the ball the negative part of $a(x)$ is not too big
(so a possibly overall negative birth rate is compensated by a sufficiently large
positive birth rate in the ball), then (\ref{0.3})
has a positive solution (see Theorem~\ref{thm 2.1} below), independently
of the size of $b(x)$. On the other hand,  the content of Theorem~\ref{thm 3.3}
is that if $a(x)$ is sufficiently small compared to $b(x)$, and certain
geometric conditions on the volume growth of the manifold hold, then (\ref{0.3}) has no
positive solution. We note that in the special case where  $b$ is
constant, and the underlying manifold is Euclidean space,
Theorem~\ref{thm 3.3} generalizes and complements
the non-existence result contained in  \cite{AB}, Section~3.

Observe also that,  according to Theorem~2.1 in \cite{BRS2},
if $a(x)$ is sufficiently  large that the bottom of the
spectrum $\lambda_1^{\Delta+a(x)}(M)$ of the Schr\"odinger operator
$\Delta + a(x)$ is negative, then one can guarantee the existence
of a (minimal) solution of (\ref{0.3}) irrespectively  of the size
of $b(x)$. This is tightly related to the above mentioned
relationship between the existence of steady state solutions and the
principal eigenvalue $\lambda_*$ of the problem $\Delta u + \lambda
a(x)u=0$ on $M$.  Indeed, as we shall explain in Section~1,
$\lambda_*$ is precisely the largest value of $\lambda$ for which
$\lambda_1^{\Delta +\lambda a(x)}(M)\geq 0.$

Note however, that in our main existence result, Theorem~\ref{thm
2.1},  we avoid an assumption of this type and describe explicit
conditions on the coefficients that  guarantee existence, thus giving
a new contribution to the subject.

It should also be stressed that having replaced Euclidean space
with a Riemannian manifold,  the behavior of the equation is now
sensitive to the geometry of the underlying space, and therefore
reflects not only the mutual relationship of the coefficients
$a(x)$ and $b(x)$, but also their respective interaction with the
geometry. From the analytic point of view this introduces new
difficulties. For instance, to prove our main non-existence result
we need to determine an asymptotic a priori upper bound for the solution
$u$, and the techniques that are usually employed in Euclidean
setting are not available. We overcome the problem via a different
approach of  geometric flavor, which may be of independent
interest (see Lemma~\ref{lemma 3.1}).

From now on we denote by $(M, \langle \, ,\,\rangle)$ a connected,
complete, non-compact Riemannian manifold of dimension $m\geq 2$.
We fix a reference point $o$ in $M,$ and denote by $r(x)$ the
Riemannian distance function from $o$, and by $B_R$ and $\bdr B_R$
the geodesic ball and sphere, respectively,  of radius $R>0$
centered at $o$. Finally, we will denote by $C$, possibly with
subscripts or superscripts, a positive constant which may vary from place
to place and that may depend  on any factor quantified
(implicitly or explicitly) before its occurrence, but not on factors
quantified afterwards.
Given functions $A$ and $B,$ defined on a set
$\Omega,$ we say that $A= O(B)$ in $\Omega$ if
there exists $C$ such that
\begin{equation*}
A(t)\leq C B(t) \quad \forall t \in \Omega.
\end{equation*}

\section{On the principal eigenvalue $\lambda_{*}$}

As mentioned in the Introduction, existence results for equation
(\ref{0.2}) typically depend on the assumption that the parameter
$\lambda$ be strictly greater than the principal eigenvalue of the
linear part of the equation.

Recall that a constant $\lambda_1$ is said to be a principal
eigenvalue for the linear equation
\begin{equation*}
\Delta u +\lambda a(x) u = 0
\end{equation*}
if for $\lambda =\lambda _1$ the equation has a positive solution.

On the other hand, in the literature on the non-compact Yamabe equation
existence results often depend on the assumption that the sign of the
bottom of the spectrum of the Schr\"odinger operator associated to the
equation be negative (see, e.g. \cite{BRS2}).

It is therefore natural to investigate the relationships between
principal and spectral eigenvalues.

Let $a(x)\in C^{\infty}(M)$ and, given a fixed radius $R$,
consider the eigenvalue problem
\begin{equation}
\label{2.1}
\begin{cases}
\Delta \vp + \lambda a(x) \vp = 0 &\text{on } \, B_R, \,\, \lambda
\in \R\\
\vp = 0 &\text{on } \, \bdr B_R.
\end{cases}
\end{equation}
If $a(x_o)>0$ for some $x_o\in B_R,$ then it is well known (see
\cite{MM}, \cite{HK}) that (\ref{2.1}) has a positive principal
eigenvalue $\lambda_1(B_R),$ which is variationally characterized
by
\begin{equation}
\label{pev} \lambda_1(R) = \inf \Bigl\{\int_{B_R} |\nabla u|^2 \, :
\, u\in H^1_o (B_R) ,\, \int_{B_R} a(x) u^2 = 1\Bigr\},
\end{equation}
and a principal positive eigenfunction $\vp$ on $B_R$ satisfying
\begin{equation}
\label{2.2}
\begin{cases}
\Delta \vp + \lambda_1(R)  a(x) \vp = 0 &\text{on } \, B_R, \\
\vp = 0 &\text{on } \, \bdr B_R.
\end{cases}
\end{equation}
We note in passing that, by the maximum principle, the condition that $a(x)$ is
positive somewhere in $B_R$ is also necessary for the existence of
a positive principal eigenvalue.

It follows from (\ref{pev}) that $\lambda_1(R)$ is a
non-increasing function of $R$, and we may set
\begin{equation}
\label{2.3}
\lambda_* = \lim_{R \to +\ty} \lambda_1(R) \geq 0.
\end{equation}

On the other hand, for $\mu\in \R,$ let  $L_\mu$ be the
operator $L_\mu = \Delta +\mu a(x)$ and denote by $\lambda_1(L_\mu, R)$
the first Dirichlet eigenvalue of $L_\mu$ on $B_R$, so that
\begin{equation*}
\lambda_1(L_\mu, R) = \inf\Bigl\{
\int_{B_R} |\nabla u|^2 -\mu a(x) u^2\, :
\, u\in H^1_o (B_R) ,\, \int_{B_R}  u^2 = 1\Bigr\},
\end{equation*}
and there exists a smooth positive eigenfunction $\psi$ of $L_\mu$
on $B_R$ satisfying
\begin{equation}
\label{ef}
\begin{cases}
L_\mu \psi = -  \lambda_1(L_\mu, R) \psi    &\text{on } \, B_R, \\
\psi = 0 &\text{on } \, \bdr B_R.
\end{cases}
\end{equation}
Again $\lambda_1(L_\mu, R)$ is a non-increasing function of $R$ and
one may define
\begin{equation*}
\lambda_1^{L_\mu}(M)=\lim_{R\to +\ty} \lambda_1(L_\mu, R),
\end{equation*}
which coincides with the bottom of the $L^2$-spectrum of $L_\mu$
in the case where the operator is essentially self-adjoint on $C_c^\ty(M)$
(this happens, e.g., if the operator $L_\mu$ is bounded from below
on $C_c^\ty$, see \cite{BdCS}, Proposition~2).

By a result of W.F.~Moss and  J.~Pieperbrink, \cite{MP}, and
D.~Fisher-Colbrie and R.~Schoen, \cite{FCS}, we have that
$\lambda^{L_\mu}_1(M)\geq 0$, if and only if there exists a
positive solution $u\in C^\ty(M)$ of
\begin{equation}
\label{2.4}
\Delta u + \mu a(x) u = 0
\end{equation}
on $M$.

We are now ready to prove the following

\begin{proposition}
\label{prop 2.5}
Let $a(x)\in C^\ty(M)$ satisfy $a(x_o)>0$ for some $x_o\in M.$
Then
\begin{equation*}
\lambda_* = \sup\{\mu\geq 0 \,:\, \lambda_1^{L_\mu}(M) \geq 0\}.
\end{equation*}
\end{proposition}

\begin{proof}
Fix $R_o$ sufficiently large that $x_o\in B_{R_o}$, and choose a
sequence $R_k$ such that $R_o<R_k\nearrow +\ty$. Denote by $\vp_k
$ the solution of (\ref{2.2}) on $B_{R_k}$ with principal
eigenvalue $\lambda_1(R_k)$, normalized with $\vp_k(x_o)=1.$ Arguing
as in the proof of \cite{J}, Theorem~1, one shows that
$\{\vp_k\}$ has a subsequence which converges locally uniformly on $M$ to a
$C^\ty$ non-negative function $\vp$ satisfying $\vp(x_o) =1$ and
\begin{equation*}
\Delta \vp + \lambda_* a(x) \vp =0 \quad \text{on }\, M.
\end{equation*}
Furthermore, by the maximum principle (see \cite{GT}, p. 35),
$\vp>0$ on $M.$ It follows from (\ref{2.4}) that $\lambda_1^{L_{\lambda_*}} (M) \geq
0$, so that
$$\lambda_* \leq \sup\{\mu\geq 0 \,: \,
\lambda_1^{L_{\mu}}(M)\geq 0\}.
$$

On the other hand, let $\mu \geq 0$ be such that
$\lambda_1^{L_{\mu}}(M)\geq 0$.  We claim that $\mu \leq
\lambda_*$ so that the reverse inequality holds in the above
formula, and the required conclusion follows. To this end, let $u$
be a smooth positive function satisfying (\ref{2.4}), and fix
$R>0$ sufficiently large that $x_o\in B_R$. Defining $w=\log u$,
it follows from
(\ref{2.4}) that
\begin{equation}
\label{2.6}
\Delta w = -\mu a(x) - |\nabla w|^2.
\end{equation}
Given any $v\in C_0^\ty(B_R)$, $v\not \equiv 0,$ we multiply both sides
of (\ref{2.6}) times $v^2$, integrate by parts and use Young inequality
to obtain
\begin{equation*}
\int_{B_R} \mu a(x) v^2 + |\nabla w|^2 v^2 = \int_{B_R} 2v <\nabla
v, \nabla w> \leq \int_{B_R} |\nabla w|^2 v^2 + |\nabla v|^2,
\end{equation*}
whence
\begin{equation*}
\mu\int_{B_R} a(x) v^2 \leq \int_{B_R} |\nabla v|^2.
\end{equation*}
Now   the variational characterization of the
principal eigenvalue shows that $\mu\leq \lambda_1(R)$ and the claim
follows from the definition of $\lambda_*.$
\end{proof}

\begin{corollary}
\label{cor 2.7}
 Let $a(x) \in C^\ty(M)$ be such that $a(x_o)>0$
for some $x_o\in M$. Then a non-negative number $\mu$ satisfies
$\mu> \lambda_*$ if and only if $\lambda_1^{L_\mu}(M)<0.$
\end{corollary}

\begin{proof}
Since in our assumptions $\lambda_*\geq 0$, we may assume that
$\mu>0.$ Assume by contradiction that $\lambda_1^{L_\mu}(M)<0$
and $\mu\leq \lambda_*.$ By definition, there exists  $R$
sufficiently large that $x_o\in B_R$ and $\lambda_1(L_\mu, R)<0$,
so that, if $\psi$ is the corresponding positive eigenfunction as in
(\ref{ef}), we have
\begin{equation*}
\int_{B_R} |\nabla \psi|^2 \leq \mu \int_{B_R} a(x) \psi^2.
\end{equation*}
In particular, the integral on the right hand side is positive,
and since $\psi\in H^1_0(B_R),$ we have
\begin{equation*}
\lambda_*\leq \lambda_1(R) \leq
\frac{\int_{B_R}|\nabla \psi|^2}{\int_{B_R} a(x)\psi^2} < \mu,
\end{equation*}
which gives the required contradiction. The reverse implication is
an immediate consequence of Proposition~\ref{prop 2.5}.
\end{proof}

We remark that statements similar to Proposition~\ref{prop 2.5}
and Corollary~\ref{cor 2.7} hold (almost trivially) in the case of a bounded domain
$\Omega$ with smooth boundary such that $a(x_o)>0$ for some $x_o\in \Omega$

We conclude this section by showing an application of the results
obtained to the case of the  Schr\"odinger operator $\Delta +
\lambda a(x)$ on $\R^m$. We assume that the positive part $a_+(x)$
of $a(x)$ does not vanish identically, so that the results
described above hold, and that it satisfies the estimate
\begin{equation*}
a_+(x) \leq \frac{k}{|x|^2}
\end{equation*}
for some positive constant $k$.
According to \cite{BRS1} Lemma~2.3, if $A(t)\leq \frac{(m-2)^2}{4t^2}$,
then the equation
\begin{equation*}
\Delta \vp + A(|x|) \vp = 0
\end{equation*}
has a positive solution $\vp$ on $\R^m$. Thus, if $\lambda k\leq
(m-2)^2/4,$ then $\vp $ satisfies
\begin{equation*}
\Delta \vp + \lambda a(x) \vp \leq 0,
\end{equation*}
which according to the above mentioned result of Fisher-Colbrie
and Schoen, \cite{FCS}, gives $\lambda_1^{\Delta + \lambda a(x)}(M)\geq 0.$
We conclude that $\lambda_*$ is strictly positive, and, in fact,
\begin{equation}
\label{p ev estimate}
\lambda _* \geq \frac{(m-2)^2}{4k}.
\end{equation}
We will come back to this in Section~3 below.

\section{Existence and uniqueness results}
The established relationship between  $\lambda_*$ and
$\lambda_1^{L_\mu}(M)$ allows us to apply to the present
situation many of the results obtained in \cite{BRS2}. In
particular, we quote the following theorem which states the
existence of minimal positive solutions of equation (\ref{0.3}).

\begin{theorem}
\label{thm BRS}
Let $a(x),$ $b(x)\in C^{0,\alpha}_{loc}(M)$ for some $0< \alpha< 1$.
Assume that $b(x)>0$, and that, having set $L= \Delta +a(x),$ we have
\begin{equation*}
\lambda_1^{L}(M)<0.
\end{equation*}
Then the equation
\begin{equation}
\label{2.0}
\Delta u +a(x) u - b(x) u^\sigma = 0, \quad \sigma >1,
\end{equation}
has a unique minimal $C^2$ positive solution.
\end{theorem}

As mentioned in the introduction, if we assume that the function $a(x)$
is positive somewhere on $M$ then the condition $\lambda_1^{L}(M)<0$
amounts to the fact $\mu=1$ is larger than the principal eigenvalue
$\lambda_*$ of the problem $\Delta u + \lambda a(x)u=0$ on $M$,
and Theorem~\ref{thm BRS} compares with the existence
results in \cite{AB}, Theorem~2.2 and \cite{DM2}, Theorem~1
(for the latter, see also the remark after Theorem~\ref{thm LTY}).

The proof of Theorem~\ref{thm BRS} uses the method of super- and sub-solutions,
and the main task is the construction of a sub-solution, which is
where the assumption on the sign of $\lambda_1^{L}(M)$ plays a
crucial role. In the main result of this Section, Theorem~\ref{thm
2.1} below, we describe conditions not expressed in terms of the
sign of $\lambda_1^{L}(M)$, for which one can guarantee existence
of a solution.

Our first  result, Theorem~\ref{thm 2.0} below, states that,
if one has a global sub-solution of (\ref{2.0}) and the set where the
non-negative coefficient $b$ is suitably small,
it is always possible to prove the existence of a maximal solution.

The idea of the proof consists in applying the method of
sub- and super-solutions to a sequence of boundary value
problems on domains which exhaust the manifold.

Since a global sub-solution of (\ref{2.0}) is given, one first needs to
find local super-solution. If $b(x)$ is strictly positive a sufficiently
large constant will do. Even if this is not the case, a super-solution
can be found provided the set where $b(x)$ vanishes is small (see
Theorem~\ref{thm LTY} below).

To apply the approximation method it is also crucial that
the approximating sequence is monotonic, and this  follows from
the next comparison result.

\begin{proposition}
\label{prop 1.28} Let $D\subset M$ be an open set with smooth
boundary $\bdr D$, and assume that $a(x)$, $b(x)$ are functions in
$C(\overline D)\cap C^{0,\alpha}_{loc}(D),$ $0<\alpha<1,$
and that  $b(x)$ is non-negative and does not vanish identically on
any connected component of $D$.  Let $u,v \in C^0(\overline D)
\cap C^2(D)$ be a positive solutions  on $D$ of
\begin{equation}
\label{1.29}
\Delta u +a(x) u - b(x) u^{\sigma} = 0
\end{equation}
and
\begin{equation}
\label{1.29bis}
\Delta v + a(x) v - b(x) v^\sigma \leq 0
\end{equation}
respectively. If $u\leq v$ on $\bdr D $ then $u\leq v $ on $D.$
\end{proposition}

\begin{proof} Since $a(x)$ has
indefinite sign, the standard comparison principle does not apply.
To circumvent this problem, set
$v_-=\delta u$ for some $\delta\in (0,1].$ Since $b(x)\geq 0$ and
$1-\sigma <0$ we have
\begin{equation*}
\Delta v_- = \delta \Delta u \geq -a(x) v_- +b(x) \delta^{1-\sigma}
v_-^{\sigma} \geq  - a(x)v_- + b(x) v_-^\sigma.
\end{equation*}
Next let $v_+=v,$ so that $v_-\leq u \leq v_+$ on $\bdr D.$ By the
monotone iteration scheme there exists a  $C^2$ solution $w$ of
(\ref{1.29}) with $w=u $ on $\bdr D,$ and $v_-\leq w\leq v_+ = v$
on $D.$ In order to conclude it is enough to show that $w=u,$ and
to this end we apply an  argument used in the proof of Lemma~2.2 in
\cite{BRS2}, which we reproduce here for the sake of completeness and
the convenience of the reader.
Let $Z$ be the vector field defined on $D$ by the formula
\begin{equation*}
Z=(w^2-u^2) \nabla \log \frac wu .
\end{equation*}
Since $u$ and $w$ are solutions of (\ref{1.29}), a direct calculation
yields
\begin{equation*}
\diver Z = b(x) (w^2-u^2)(w^{\sigma -1}-u^{\sigma-1}) +
\Bigl| \frac wu \nabla u -\nabla w\Bigr|^2 +
\Bigl| \frac uw \nabla w -\nabla u\Bigr|^2.
\end{equation*}
Integrating over $D$, applying the divergence theorem, and using
the fact that $u\equiv w$ on $\bdr D$ yield
\begin{equation}
\label{1.30}
\int_{D}  \Bigl| \frac wu \nabla u -\nabla w\Bigr|^2 +
\Bigl| \frac uw \nabla w -\nabla u\Bigr|^2  =
-\int_D
b(x) (w^2-u^2)(w^{\sigma -1}-u^{\sigma-1})\leq 0.
\end{equation}
It follows that $\nabla w - \frac w u \nabla u=0$ so that $u=Bw$
on any connected component $D_1$ of $D,$ for some constant $B>0.$
Inserting this into the inequality
\begin{equation*}
-\int _{D_1} b(x) (w^2-u^2)(w^{\sigma -1}-u^{\sigma-1})\leq 0
\end{equation*}
yields
\begin{equation*}
(1-B^2)(1-B^{\sigma -1}) \int_{D_1} b(x)w^{\sigma +1} \leq 0.
\end{equation*}
Since $w>0$ and $b\geq 0$,  $b\not\equiv 0$ on $D_1$ this forces
$B=1.$ Thus $u=w$ on $D_1,$ and therefore $u=w$ on $D, $ as required.
\end{proof}

We remark that Proposition~\ref{prop 1.28} holds if the
coefficients $a(x)$ and $b(x)$ are only assumed to be continuous, and if
the functions $u$ and $v$ are in $C^1(D)\cap C^0(\overline D)$,
provided we interpret (\ref{1.29}) and (\ref{1.29bis}) in weak sense.
Under these weaker assumptions, the vector field $Z$ will be only continuous, in general,
but the proof may be carried out using a suitable version of the divergence
theorem (see, e.g., \cite{RS}, pp. 477--478).

Before stating  Theorem~\ref{thm 2.0}, we also need to
make precise the sense in which the set where $b(x)$ vanishes is small.

Let $a(x) \in C^0(M)$ and let  $L=\Delta +a(x)$. If $\Omega$ is a
non-empty open set, the first Dirichlet eigenvalue
$\lambda_1^L(\Omega)$ is variationally defined as in  Section 1 by
means of the formula
\begin{equation*}
\lambda_1^L(\Omega) =
\inf\Bigl\{ \int_{\Omega} |\nabla \phi|^2 - a(x) \phi^2\, :
\, \phi\in H^1_o (\Omega) ,\, \int_{\Omega}  \phi^2 = 1\Bigr\},
\end{equation*}
and, if $\Omega$ is bounded and both $\Omega$ and $a$ are sufficiently
regular, the infimum is attained and there exists a unique normalized
eigenfunction $v$ defined on
$\Omega$  satisfying
\begin{equation*}
\begin{cases}
\Delta v + a(x) v + \lambda_1^L (\Omega) v = 0 \quad \text{ on }
\,\, \Omega &\\
v>0 \,\text{ on } \, \Omega, \,\,\, v\equiv 0 \, \text{ on }\,
\bdr \Omega.
\end{cases}
\end{equation*}
We extend the definition to an arbitrary  bounded subset $S$ of
$M$, by setting
\begin{equation*}
\lambda_1^L(S) = \sup \lambda_1^L(\Omega),
\end{equation*}
where the supremum is taken over all open bounded sets with smooth boundary
$\Omega$ such that $S\subset \Omega$. Note that, by definition, if
$S=\emptyset$ then $\lambda_1^L(S)=+\infty.$
Finally, if $S$ is an
unbounded subset of $M$, we define
\begin{equation*}
\lambda_1^L(S) = \inf \lambda_1 (D\cap S),
\end{equation*}
where the infimum is taken over all bounded open sets with smooth
boundary.  Note that if $\{D_n\}$ is a increasing sequence of open sets
with smooth boundary which exhausts $M,$ then, by domain monotonicity,
$\lambda_1^L(S) = \lim_n\lambda_1 (D_n\cap S).$

Since the first Dirichlet eigenvalue of the Laplacian of a ball $B_r$
grows like $r^{-2}$ as $r\to 0,$ $\lambda_1^L(B_r)>0$ provided $r$ is
sufficiently small, and one may think that the condition $\lambda_1^L(S)>0$
expresses the fact that $S$ is small in a spectral sense.

This  notion of smallness is appropriate for our purposes.
Indeed,  P.~Li, L.-F.~Tam and D.~Yang, \cite{LTY},
have established the following relationship between
the first eigenvalue of the set where $b(x)$ vanishes, and the
existence of a non trivial super-solution of equation (\ref{2.0}).

\begin{theorem}
\label{thm LTY}
Let $a(x)$ and $b(x)$ be H\"older continuous functions on $M$,
with $b(x)\geq 0$ on $M$, and let $S_o=\{x \in M \, :\, b(x)=0\}$.
Let $\Omega$ be a bounded open domain in $M$, and let $L=\Delta+a(x).$
If equation (\ref{2.0}) has a positive super-solution on $\Omega,$ then
$\lambda_1^L(\Omega\cap S_o) \geq 0$.
\newline
Conversely, if $\lambda_1^L(\Omega\cap S_o) > 0$ then (\ref{2.0})
has a positive super-solution on $\Omega.$
\end{theorem}

We remark that using  Theorem~\ref{thm LTY}  we may improve
Theorem~\ref{thm BRS} above replacing the assumption that $b(x)>0$ on M with the
assumption that $b$ is non-negative, and its zero set $S_o$ is
such that $\lambda_1^L(S_o)>0$. Indeed,
the strict positivity of $b$ is only used to guarantee that
on every bounded domain a suitably large constant is a
super-solution of equation (\ref{2.0}) (see \cite{BRS2}, p.184).

Note that, if we assume that  $S_o$ is a bounded domain with smooth boundary such that
$a(x)$ is positive somewhere in $S_o$, then the condition $\lambda_1^L(S_o)>0$
amounts to the fact that $1$ is strictly smaller that the principal
eigenvalue of the problem $\Delta u + \lambda a(x) u =0$ with
Dirichlet boundary conditions.  Recalling  the remark
after the statement of Theorem~\ref{thm BRS} we conclude that if
$\lambda_* < 1 < \lambda_1(L_\mu, S_o)$ then equation (\ref{2.0})
has a unique positive minimal solution on $M$. This again compares with Theorem~1
in \cite{DM2}.

We are now ready to state

\begin{theorem}
\label{thm 2.0} Let $a(x),$ $b(x)\in C^{0,\alpha}_{loc}(M)$  for
some $0<\alpha\leq 1$. Assume that $b(x)$ is non-negative, and
strictly positive off a compact set, and that, denoting with $S_0$
the set where $b(x)$ vanishes, we have $\lambda_1^L(S_o)>0$. If
$u_-\in C^0(M)\cap H^1_{loc}(M)$, $u_-\geq 0,$ $u_-\not\equiv 0$,
is a global sub-solution of equation (\ref{2.0}), then
(\ref{2.0}) has a maximal positive $C^2$ solution.
\end{theorem}

\begin{proof}
Let $D_k$ be an increasing exhaustion of $M$ by open domain with
smooth boundary, such that $S_0 \subset D_k \subset \bar D_k
\subset D_{k+1}$ for every $k$. Fix $k$ in $\N.$ Since
$\lambda_1^L(S_o)>0$, by Theorem~\ref{thm LTY} there exists a
$C^2$ positive function $v$ satisfying
\begin{equation*}
\Delta v + a(x) v - b(x) v^\sigma \leq 0 \text{ on } \, D_{k+1}.
\end{equation*}
Since $\bar D_k$ is compactly contained in $D_{k+1},$ $\inf_{D_k}
v>0$ and $u_-$is bounded  on $\bar D_k$. Thus, given
$n\geq \max_{\bar D_k} u_-,$ there exists $C>0$ large enough that
the function $v_+=Cv$ satisfies
\begin{equation*}
\begin{cases}
\Delta v_+ + a(x) v_+ - b(x) v_+^\sigma \leq 0 \quad \text{ on }\, D_k&\\
v_+\geq n\geq \max_{\bar D_k} u_- \,\text{ on } \, \bdr D_k
\,\text{ and } \, v_+ \geq u_- \,\text{ on }\, \bar D_k.&
\end{cases}
\end{equation*}
The monotone iteration scheme yields a solution $u_{k,n}$ of the
boundary value problem
\begin{equation}
\label{1.26}
\begin{cases}
\Delta u +a(x) u - b(x) u^{\sigma} = 0 &\text{on } B_k\\
u=n &\text{on } \, \bdr B_k.
\end{cases}
\end{equation}
We now show that the sequence $\{u_{k,n}\}$  is uniformly bounded
with respect to $n\in \N$ on compact subsets of $D_k.$

Assume first that $K$ is a compact subset of $D_k$ which does not
intersect $S_0$. Then we may find a positive constant $b_o$ and a
finite number of disjoint open balls $B_i$ which cover $K$ such
that $b(x)\geq b_o$ on each $B_i$. Applying Lemma~2.6 in
\cite{PRS2} we deduce that there exists a constant $C_1=C_1(K)>0$
such that
\begin{equation}
\label{ub in K_1} u_{k,n}(x)\leq C_1 \quad \forall x\in
K,\,\,\forall n.
\end{equation}

Next we show that $u_{k,n}$ is uniformly bounded in a neighborhood
of $S_0$. By definition there exist  open sets with smooth
boundary $\Omega$ and $\Omega'$ such that $S_0\subset
\Omega\subset \bar \Omega \subset\Omega' \subset \bar \Omega'
\subset D_k$ and $\lambda_1^L(\Omega')>0$.

Note that since $\bdr \Omega$ is a compact subset of $D_k$ which
does not intersect $S_0$, there exists a constant $C_2$ such that
$u_{k,n}\leq C_2$ on $\bdr \Omega$. Next, let $\phi$ be a positive
eigenfunction for $L$ belonging to $\lambda^L_1(\Omega')>0$. Since
$\phi$ is positive on $\Omega'$, it is bounded away from zero on
$\bar \Omega$ and there exists a positive constant $c$ such that
$c\phi
> C_2$ on $\bar \Omega$.

Note that
\begin{equation*}
\Delta(c\phi ) + a(x) (c\phi) = - \lambda_1^L(\Omega'')(c\phi)< 0
,
\end{equation*}
while
\begin{equation*}
\Delta u_{k,n} + a(x) u_{k,n} = b(x)  u_{k,n}^\sigma \geq 0
\end{equation*}
on $\Omega$, and $u_{k,n}\leq C_2< c \phi$ on $\bdr \Omega,$

We claim that $u_{k,n}\leq c \phi$ on $\Omega$. Indeed, assume
that this is not so, and let $A=\{x\in \Omega' \, : \, u_{k,n} -
c\phi >0\}$. Then $A$ is non-empty and $\bar A \subset \Omega$,
and we deduce that $w=u_{k,n} -c \phi$ attains a positive maximum
in $A.$ On the other hand $w$ satisfies
\begin{equation*}
\begin{cases}
\Delta w + a(x) w \geq 0 &\text{in }\, A  \\
w= 0 &\text{on }\,\bdr A,
\end{cases}
\end{equation*}
and therefore, by the generalized maximum principle, $w/\phi$ is
constant on $A,$ and since it vanishes on $\bdr A$ we conclude
that $w/\phi = 0$ on $A$, that is, $w=0$ on $A$, contradiction.

Thus $u_{k,n} \leq c\phi\leq C_2$ on $\bar \Omega$ and it follows
easily that $u_{k,n}$ is uniformly bounded on compact subset of
$D_k$.

By  interior elliptic estimates, a subsequence of $u_{k,n}$
converges in $C_{loc}^2$ to a solution $u_k^\infty$ of
\begin{equation}
\label{1.27}
\begin{cases}
\Delta u +a(x) u - b(x) u^{\sigma} = 0 &\text{on } D_k\\
u=+\ty  &\text{on } \, \bdr D_k.
\end{cases}
\end{equation}
We consider the sequence $\{u_k^\ty\}$. Clearly, $u_k^\ty\geq
u_->0$, and an exhaustion argument and  Proposition~\ref{prop
1.28} show that
\begin{equation}
\label{1.28} u_{k+1}^\ty \leq u_k^\ty \qquad \text{ on }\,\,
\overline{D}_k.
\end{equation}

Since $\{u_k^\ty\}$ is monotone non-increasing, it converges  to a
function $u$ which solves (\ref{2.0}) and satisfies $u\geq u_-
\geq 0$, $u_-\not\equiv 0$ on $M$. If $u_1$ is another positive solution
of (\ref{2.0})
on $M$, then $u_1\leq u_k^\ty$ by Proposition~\ref{prop 1.28}, and
therefore $u_1 \leq u$, thus proving the maximality of $u$.
Finally, $u$ is strictly positive for otherwise the non-negative function $u_-$
would attain a zero minimum, thus violating  the minimum  principle (\cite{GT}, p.
35).
\end{proof}

\par
It is worth pointing out the following consequence of
Proposition~\ref{prop 1.28}.

\begin{proposition}
\label{prop 1.32} Let $a(x)$, $b(x)\in C^{0}(M)$,
$0<\alpha<1$,  and $b(x)\geq 0,$
$b(x)\not\equiv 0$. Then the problem
\begin{equation}
\label{1.33}
\begin{cases}
\rmi \, \, \Delta u + a(x) u - b(x) u^\sigma =0, \quad \sigma >1 &\\
\rmii\, \lim_{r(x)\to +\ty} u(x) = L>0 &
\end{cases}
\end{equation}
has at most one positive $C^2$ solution.
\end{proposition}

\begin{proof}
Let $u,$ $v$ be positive $C^2$ solutions of (\ref{1.33}). Choose
any $\e>0$ and observe that, since $b$ is non-negative, the
function $w_\epsilon = (1+\epsilon) v$ is a super-solution of
(\ref{1.33}) \rmi satisfying
\begin{equation}
\label{1.34}
\lim_{r(x) \to +\ty} w_\epsilon (x) = (1+\epsilon)L.
\end{equation}
Fix $R_0>0$ sufficiently large, so that for every $R>R_0$ we have
$b(x)\not\equiv 0$ on $B_R$ and $w_\epsilon - u>0$ on $\bdr B_R$.
The latter is possible because of the limit relations
(\ref{1.33}) \rmii and (\ref{1.34}). It follows from
Proposition~\ref{prop 1.28} that $w_\epsilon \geq u$ on $B_R$ for
every $R\geq R_0$. Thus, $u(x)\leq (1+\epsilon) v(x)$ on $M$, and
since $\epsilon>0$ was arbitrary, $u\leq v$ on $M.$ Interchanging
$u$ and $v$ yields the reverse inequality, and equality follows.
\end{proof}

We remark that the assumption $L>0$ in the statement of the
proposition cannot be weakened to $L\geq 0$. Indeed, in
\cite{BRS2}, pp. 214--215, it is shown that on $m$-dimensional
hyperbolic space $\mathbb{H}^m,$ equation (\ref{1.33}) with
$a(x)\equiv m(m-2)/4$, $b(x) \equiv 1$  and $\sigma= (m+2)/(m-2)$
has a family of positive distinct radial solutions
which tend to zero at infinity at the same rate.

It may also be worth noting that  the assumptions on $u$ and $v$ in the
uniqueness result obtained above may be weakened provided some conditions on
the coefficients $a$ and $b$ and on the volume growth of the manifold are
imposed.

\begin{theorem}
\label{thm comparison}
Let $a(x),$ $b(x)\in C^0(M)$ and assume that, for some $C>0$ and $0\leq \mu<2,$
\begin{equation}
\label{b lower estimate}
\rmi \, b(x) \geq C (1+r(x))^{-\mu} \quad\text{and}\quad
\rmii \, \sup_M
\frac{a_{-}}{b} <+\infty.
\end{equation}
Assume that $u$ and $v$ are $C^2$ nonnegative solutions of
\begin{equation*}
\Delta u +a(x) u - b(x) u^{\sigma} \geq 0\geq
\Delta v + a(x) v - b(x) v^\sigma
\quad \sigma >1,
\end{equation*}
on $M,$ satisfying
\begin{equation*}
\liminf_{r(x)\to \infty} v(x)>0,
\quad
\limsup_{r(x)\to \infty} u(x) <+\infty.
\end{equation*}
If
\begin{equation}
\label{vol growth}
\liminf_{r\to +\infty} \frac{\log \vol B_r}{r^{2-\mu}}<+\infty,
\end{equation}
then $u\leq v.$
\end{theorem}
\begin{proof}
Note first of all that, by the maximum principle (see \cite{GT},
p. 35), $v$ is strictly positive, and therefore,  by the liminf
condition,  it is bounded away from $0$ on $M$. Also,  $u$
is bounded above on $M.$
We may assume that $u$ is not identically zero, for else there is
nothing to prove. Thus, $\beta=\sup_M \frac u v $  is finite  and
strictly positive. The conclusion of the theorem amounts to saying that
$\beta\leq 1$. Assume by contradiction that $\beta>1,$ and let
$\phi = u-\beta v$.

Clearly $\phi\leq 0$. We claim that $\sup _M\phi =0.$
Indeed, there exists a sequence $x_n$ such that $\frac{u(x_n)}{v(x_n)}\to
\beta >0,$ and since $u(x_n)$ is bounded above, so must  be
$v(x_n)$ (for else $\beta=0$), and then
\begin{equation*}
\phi(x_n) = v(x_n) \bigl(\frac{u(x_n)}{v(x_n)}-\beta\bigr) \to 0
\quad \text{as}\,\, n\to +\infty,
\end{equation*}
as claimed.

We write
\begin{equation}
\label{lapl phi}
\begin{split}
\Delta \phi &\geq
 \Delta ( u-\beta v) = -a(x) [u-\beta v] + b(x)
[u^\sigma -\beta v^\sigma] \\
&= - a(x) \phi + b(x) [u^\sigma - (\beta v)^\sigma] + b(x)
v^\sigma [\beta^\sigma -\beta].
\end{split}
\end{equation}
By the mean value theorem
we have
\begin{equation*}
[u^\sigma - (\beta v)^\sigma](x) = h(x)(u-\beta v) = h(x) \phi
\end{equation*}
where
\begin{equation*}
h(x) = \frac{\sigma}{u(x) - \beta v(x)}\int_{\beta v(x)}^{u(x)}
t^{\sigma -1} dt
\end{equation*}
is continuous, and  nonnegative on $M.$ Further, since $u$ is bounded above,
it follows $h$ is bounded above by a constant $H$ on the set
$\{x : \phi(x) > -1\}.$
Also, since $\beta>1$, $\sigma>1$ and $v$ is bounded away from
zero,
\begin{equation*}
v^\sigma [\beta^\sigma -\beta] \geq
2c >0,
\end{equation*}
for some positive constant $c.$
Inserting the above expressions in (\ref{lapl phi}), noting that,
since $\phi$ is non-positive, $-a(x) \phi \geq a_-(x) \phi$,
and dividing through by $b(x),$
we obtain
\begin{equation*}
\frac 1{b} \Delta \phi \geq \bigl( \frac{a_-} b + h(x) \bigr) \phi
+ 2c.
\end{equation*}
Now let $\epsilon\in (0,1)$ be such that
\begin{equation*}
(\sup\frac {a_-}{b} + H) \epsilon < c,
\end{equation*}
and let $\Omega_\epsilon = \{x : \phi (x) >-\epsilon\}$, which
is not empty since $\sup \phi =0.$ Then
\begin{equation*}
\frac1{b(x)} \Delta \phi \geq c >0 \quad \text{on }
\,\Omega_\epsilon.
\end{equation*}

On the other hand, since the  volume growth condition (\ref{vol growth})
holds, and $b(x)$ satisfies the lower estimate   (\ref{b lower estimate}) (i),
Theorem~A in \cite{PRS1} applies, and the weak maximum principle holds,
namely,
$$
\inf_{\Omega_\epsilon} \frac 1{b(x)} \Delta \phi \leq 0
$$
thus yielding the  required contradiction.
\end{proof}

As an immediate corollary we have

\begin{corollary}
\label{cor uniqueness}
Let $a$ and $b$ satisfy the conditions listed in
Theorem~\ref{thm comparison}, and let
$u$ and $v$ be  nonnegative solutions of
\begin{equation*}
\Delta u +a(x) u -b(x) u^\sigma=0.
\end{equation*}
If both $u$ and $v$ satisfy the condition
\begin{equation}
\label{liminf limsup condition}
0<\liminf_{r(x)\to \infty} u(x) \leq \limsup_{r(x)\to \infty} u(x)
<+\infty,
\end{equation}
and  (\ref{vol growth}) holds, then $u=v$.
\end{corollary}

As observed above,  condition (\ref{liminf limsup condition})
amounts to requiring that $u$ and $v$ are bounded and bounded
away from zero on $M.$ We also note that the family of functions
mentioned in the remark that follows Proposition ~\ref{prop 1.32}
also  shows that uniqueness fails if we do not assume that
the liminf of $u$ and $v$ are strictly positive.

It is worth mentioning the following geometric consequence.

\begin{corollary}
\label{cor Schwarz}
Let $(M, \langle \, , \, \rangle)$ be a complete Riemannian
manifold of dimension $m\geq 3$ and scalar curvature $s(x)$
satisfying
\begin{equation*}
s(x) \leq - C (1+r(x))^{-\mu}
\end{equation*}
for some constants $C>0$ and $0\leq \mu<2.$ Assume that (\ref{vol
growth}) holds. Then any conformal diffeomorphism of $M$ into
itself which preserves the scalar curvature and whose stretching
factor $u$ satisfies (\ref{liminf limsup condition}), is an isometry.
\end{corollary}

\begin{proof}
Let $\phi: M\to M$ be a conformal diffeomorphism. Then $\phi^\star \langle \, ,\,\rangle =
u^{\frac 4{m-2}} \langle \,,\,\rangle$ where $u$ is the stretching
factor. Since $\phi$ preserves the scalar curvature, $u$ is a
solution of
\begin{equation*}
c_m \Delta u - s(x) u + s(x) u^{\frac{m+2}{m-2}} = 0,
\end{equation*}
(see, e.g., \cite{PRS1}, p. 1319 ff), and the result follows at
once from Corollary~\ref{cor uniqueness}.
\end{proof}

We remark that here we require that $u$ is bounded above and away
from zero, so that the conformal diffeomorphism $\phi$ is a
quasi-isometry. By contrast, in \cite{PRS1} Corollary 3.4, $\phi$
is not assumed to be a quasi-isometry, but  the scalar curvature $s(x)$ is
assumed to be bounded below.

We now proceed with the main result of this section,
Theorem~\ref{thm 2.1}, where we show that, under suitable assumptions
on the coefficients, equation (\ref{2.0}) has a globally defined
positive sub-solution, and therefore, a maximal positive solution.
The proof is based on the method of super and sub-solutions. This is
achieved by constructing  a sub-solution inside and outside a suitable
ball in such a way that they can be glued together to yield a global
sub-solution.

We begin with the following lemma, which will be the key ingredient in
the construction of a sub-solution in the complement of a ball.

\begin{lemma}
\label{lemma 1.1new} Let $A(r),$ $B(r)\in C^0([0, +\infty))$ with
$A(r)\geq 0$ and $B(r)>0$ on $[0,+\infty)$. Let  $g$ be a
non-decreasing smooth function on $[R, +\infty)$,  for some $R>0$.
 Then, given $T>0$ and $\sigma>1$ , the
problem
\begin{equation}
\label{radial annulus}
\begin{cases}
\alpha'' + (m-1) \ds{\frac{g'}g} \alpha' - A(r) \alpha -
B(r)\alpha^\sigma =0 \, \text{ on }\, (R,R+T)&\\

\alpha(R) = \alpha_o , \,\, \alpha(R+T)= 0
\end{cases}
\end{equation}
has a $C^2$  solution $\alpha$ on $[R, R+T]$. Furthermore
$\alpha >0$ and $\alpha'<0$ in $[R,R+T)$ and for every $T_o\in (0,T]$ the
following estimate holds
\begin{equation}
\label{1.3}
|\a'(R)| \leq
\Bigl\{\frac{g(R+T_o)}{g(R)}\Bigr\}^{m-1}
\Bigl\{
T_o \max_{[R,R+T_o]} [A(s) + B(s)\alpha_o^{\sigma -1}]
+ \frac 1{T_o}
\Bigr\} \alpha_o.
\end{equation}
\end{lemma}

\begin{proof}
We first show the existence of a solution $\alpha$ to (\ref{radial
annulus}). We extend $g$ to a smooth, non-decreasing function on
$[0, +\infty)$ satisfying $g'(0)=1$ and $g^{(2k)}(0)=0$ for every
$k\in \N$,  and consider the model manifold $M=\R^m$ with the
metric given in polar coordinates by
\begin{equation*}
\langle \,,\,\rangle = dr^2 + g(r)^2 d\theta ^2.
\end{equation*}
Note that by the conditions imposed on the function $g$, the
metric originally defined on $\R^m\setminus\{0\}$ extends to a
smooth metric on the whole of $\R^m.$

Let $R_1$ and $\bar R$ be such that $0<R_1<R< R+T<\bar R$, and
let $\psi$ be a smooth radial cut-off function such that $0\leq \psi\leq 1,$ $\psi\equiv 1$ on
$B_{R_1},$ and $\psi \equiv 0$ on $M\setminus \bar B_{R+T}$. Define
\begin{equation*}
\bar a (x) = -\psi(x) A(r(x)) + N(1-\psi(x))
\end{equation*}
where $N$ is  constant and let $\bar L$ be the Schr{\"o}dinger operator
$L=\Delta + \bar a(x)$. We claim that if $N$ is sufficiently large
then
\begin{equation*}
\lambda_1^{\bar L} (B_{\bar R})<0.
\end{equation*}
Indeed, let $u$ be any smooth function satisfying $u>0$ in $B_{\bar
R}$ and $u=0$ on $\bdr B_{\bar R}$. Then
\begin{equation*}
\int_{B_{\bar R}} |\nabla u |^2 - \bar a(x) u^2 = \int_{B_{\bar R}} [ |\nabla u |^2 +
\psi(x) A(r(x)) u^2 ] - N \int_{B_{\bar R}} (1-\psi) u^2,
\end{equation*}
and since $(1-\psi ) u^2 >0$ in $B_{\bar R} \setminus B_{R+T}$,
the right hand side may be made negative provided $N$ is large enough.
The claim now follows from the variational characterization of
$\lambda_1^{\bar L}(B_{\bar R}).$

Let $\phi$ be the radial, normalized eigenfunction belonging to
$\lambda_1^{\bar L}(B_{\bar R})$. By definition
\begin{equation*}
\int_{B_{\bar R }}  |\nabla \phi|^2 - \bar a(x) \phi^2 =
\lambda_1^{\bar L}(B_{\bar R})<0,
\end{equation*}
so that, if $\gamma>0$ is sufficiently small,
\begin{equation*}
\int_{B_{\bar R }}  |\nabla \phi|^2 - \bar a(x) \phi^2
+\gamma   B(r(x)) \phi^2 = \lambda_1^{\bar L}(B_{\bar R})
+\gamma \int_{B_{\bar R }}  B(r(x)) \phi^2
<0.
\end{equation*}
Thus, if  we denote by $\tilde L = L+\bar a(x) -\gamma B(r(x))$,
then
\begin{equation*}
\lambda_1^{\tilde L} (B_{\bar R})< 0.
\end{equation*}
If $\xi$  is a positive radial eigenfunction belonging to
$\lambda_1^{\tilde L} (B_{\bar R})$, so that
\begin{equation*}
\Delta \xi + \bar a(x) \xi = \gamma B(r(x)) \xi -
\lambda_1^{\tilde L} (B_{\bar R})\xi \geq \gamma B(r(x)) \xi
\,\,\text{ on } \, B_{\bar R},
\end{equation*}
then the function $v_-= \mu \xi$  satisfies
\begin{equation*}
\begin{cases}
\Delta v_- + \bar a(x) v_- - B(r(x)) v_-^\sigma
\geq B(r(x)) \mu\xi \bigl[\gamma - (\mu \xi)^{\sigma -1}\bigr]
\geq 0\,  \text{ in }\, B_{\bar R} &\\
v_->0 \,\text{ on }\,B_{\bar R} \quad
 v_- \equiv 0 \,\text{ on }\, \bdr B_{\bar R}&
\end{cases}
\end{equation*}
provided $0<\mu< \gamma^{1/(\sigma -1)}( \sup_{B_{\bar R}}\xi )^{-1}$.

On the other hand, since $B(r)>0,$ a sufficiently large positive constant
$v_+$ is a super solution of the above problem, that is
\begin{equation*}
\begin{cases}
\Delta v_+ \bar a(x) v_+ - B(r(x)) v_+^\sigma\leq 0
\,  \text{ in }\, B_{\bar R} &\\
v_+>0 \,\text{ on }\, \bdr B_{\bar R},&
\end{cases}
\end{equation*}
and by the monotone iteration scheme there exists a solution $w$
to the problem
\begin{equation*}
\begin{cases}
\Delta w + \bar a(x) w - B(r(x)) w^\sigma= 0, \, w>0
\,  \text{ in }\, B_{\bar R} &\\
w\equiv 0 \,\text{ on }\, \bdr B_{\bar R}.&
\end{cases}
\end{equation*}
Note that $w$ is radial, since the monotone iteration scheme
produces radial solutions in radial setting.
Moreover, since $\bar a(x) > - A(r(x))$ and $\inf_{\bar B_{R+T}} w>0$,
if $c$ is sufficiently large then  the function  $w_+ = cw$ satisfies
\begin{equation*}
\begin{cases}
\Delta w_+ -A(r(x)) w_+ - B(r(x)) w_+^\sigma\leq 0
\,  \text{ in }\, B_{\bar R} &\\
w_+ \geq  0 \,\text{ on }\, \bdr B_{R+T}\cup \bdr B_R,&
\end{cases}
\end{equation*}
and since $w_-\equiv 0$ is a sub-solution of the problem, applying
once again the monotone iteration scheme produces a radial non-negative
$C^2$ solution  $u$ of
\begin{equation*}
\begin{cases}
\Delta u+ -A(r(x)) u - B(r(x)) u^\sigma=  0
\,  \text{ in }\, B_{ R +T}\setminus \bar B_R &\\
u\geq 0, \, u=\alpha_o \,\text{ on }\, \bdr B_{R},\, u= 0\text{ on }
\bdr B_{R+T}.&
\end{cases}
\end{equation*}
Since $u$ satisfies
\begin{equation*}
\Delta u -\bigl( A(r(x)) + B(r(x)) u^{\sigma -1}\bigr) u=0,
\end{equation*}
by the strong maximum principle (see \cite{GT}, p. 35) we deduce
that $u>0$ in $B_{ R +T}\setminus \bar B_R$. Since $u$ is radial,
we deduce that $u(x)=\alpha(r(x))$ with $\alpha$ satisfying
(\ref{radial annulus}).

We now show that $\alpha'<0$ in $[R,R+T)$. Indeed, assume this is not
so. Since $\alpha(r)>0=\alpha(R+T)$,  then there exists $r_o\in[R,
R+T)$ such that $\alpha'(r_o)=0$
Integrating (\ref{radial annulus}) between $r_o$ and $r$ gives
\begin{equation*}
g(r) ^{m-1} \alpha'(r) = \int_{r_o} ^r g(s)^{m-1} \alpha(s)
[A(s) + B(s) \alpha(s)^{\sigma-1}] \, ds
\end{equation*}
and the integral on the right hand side is striclty positive for
every $r>r_o$. Thus $\alpha'(r)>0$ in $[r_o, R+T],$ contradicting
$\alpha(r_o)>0$, $\alpha(R+T)=0.$

It remains to show that estimate (\ref{1.3}) holds.
Integrating (\ref{radial annulus})
between $R+t_1$ and $R+t_2$ with  $0\leq t_1<t_2\leq T$, yields
\begin{multline}
\label{a' estimate 1}
g(R+t_2) ^{m-1} \alpha'(R+t_2) -g(R+t_1) ^{m-1} \alpha'(R+t_1)
\\
= \int_{R+t_1} ^{R+t_2} g(s)^{m-1} \alpha(s)
[A(s) + B(s) \alpha(s)^{\sigma-1}] \, ds,
\end{multline}
and since the integrand is positive, we deduce that
\begin{equation*}
g(R+t_1) ^{m-1} \alpha'(R+t_1) < g(R+t_2) ^{m-1}
\alpha'(R+t_2),
\end{equation*}
whence, recalling that  $\alpha'<0$ and $g$ is non-decreasing,
\begin{equation*}
\frac{\alpha'(R+t_1)}{\alpha'(R+t_2)} > \left(
\frac{g(R+t_2)}{g(R+t_2)}\right)^{m-1} \geq 1,
\end{equation*}
and therefore
\begin{equation*}
\alpha'(R+t_1)< \alpha'(R+t_2)\leq 0 \quad \forall 0\leq t_1<t_2\leq
T.
\end{equation*}
Moreover, for every $T_o\in (0,T]$ there exists $t\in (0, T_o)$ such that
\begin{equation*}
\frac{ \alpha_o}{T_o} = \frac{\alpha (R) - \alpha(R+T)} {T_o} =
-\alpha'(R+t)>-\alpha(R+T_o).
\end{equation*}
Writing (\ref{a' estimate 1}) with $t_1=0$ $t_2=T_o$, and using the above
inequality  we obtain
\begin{equation*}
-g(R)^{m-1} \alpha'(R)
\leq
g(R+T_o)^{m-1} \frac {\alpha_o}{T_o} + T_o g(R+T_o)^{m-1} \alpha_o
\max_{[R,R+T_o]}[A(s) + B(s) \alpha_o^{\sigma -1}],
\end{equation*}
whence, for every $t\in[R,R+T]$,
\begin{equation*}
|\alpha'(t)|\leq |\alpha'(R)|\leq \left( \frac
{g(R+T_o)}{g(R)}\right)^{m-1} \left\{ \frac 1{T_o} + T_o
\max_{[R,R+T_o]}[A(s) + B(s) \alpha_o^{\sigma -1}]\right\}\alpha_o,
\end{equation*}
as required.
\end{proof}

We are now ready to prove the main result of this section.

\begin{theorem}
\label{thm 2.1}
Let $(M,\langle \,,\,\rangle)$ be a complete manifold, and assume
that the differential inequality
\begin{equation}
\label{1.5}
\Delta r(x) \leq (m-1) \frac{g'(r(x))}{g(r(x))}
\end{equation}
holds pointwise in the complement of the cut locus $Cut_o$ of $o$,
for some function $g\in C^\ty([0,+\ty))$, with $g^{(2k)}(0)=0$ if
$k=0,1,2,\dots,$ $g'(0)=1,$  $g(t)>0$ and $g'(t)\geq 0$ if $t>0.$
Let $a(x),$ $b(x)\in C^{0,\alpha}_{loc}(M)$ ($0<\alpha\leq 1$), and assume that
$b(x)$ satisfies the conditions in Theorem~\ref{thm 2.0}.
Suppose that we can choose $R>0$ and $T_o>0$
in such a way that
\begin{equation}
\label{1.19}
R\, \inf_{B_R} a> (1+\sup_{B_R} r\Delta r)
\bigl(\frac 1{T_o}
+ T_o\max_{\overline{B}_{R+T_o}\setminus B_R} a_-
\bigr)
\biggl(
\frac{g(R+T_o)}{g(R)}
\biggr)^{m-1}.
\end{equation}
Then  the equation
\begin{equation}
\label{1.20}
\Delta u+ a(x) u - b(x) u^\sigma = 0, \quad \sigma>1,
\end{equation}
has a  maximal positive solution on $M$.
\end{theorem}

\par
\noindent {\bf Remarks} Observe that, by the Laplacian Comparison
Theorem, see, e.g. \cite{GW}, or \cite{BRS2}, Appendix, the
validity of an inequality of the form (\ref{1.5}) can be deduced
from appropriate lower bounds for the radial Ricci curvature. To
compare with the existence theorems in Euclidean setting that can
be found in the literature, we note that, on $\R^m$, (\ref{1.5})
holds with $g(r)=r$, while if we assume an "almost Euclidean
behavior", namely, that the radial Ricci curvature satisfies an
estimate
\begin{equation*}
Ricc\geq - (m-1)B(1+r^2)^{-1},
\end{equation*}
then (\ref{1.5}) holds with $g(r)= r^{B'}$ where
$B'=[1+\sqrt{1+4B^2}]/2$. We note in passing that the assumption
on the Ricci curvature
does not imply that the manifold is quasi-isometric to Euclidean
space.
If we assume instead that
\begin{equation*}
Ricc\geq - (m-1)B
\end{equation*}
so that, loosely speaking, the reference model is  hyperbolic
space of constant negative sectional curvature, then(\ref{1.5})
holds with the choice $g(r)= \frac 1{\sqrt B} \sinh(\sqrt B r)$.
\par

We also note that condition (\ref{1.19}) is satisfied if $a(x)$ is
sufficiently large near the  origin $o$ and non-negative in a
suitably large ball. Further, up to choosing a different reference point,
it suffices to  assume that $a(x)$ has a positive spike somewhere. Of course
(\ref{1.5}) should then be written with respect to the new origin centered
at the spike.
Observe however that if we assume, e.g., that  the Ricci curvature
is bounded from below by a negative constant, then the validity of
(\ref{1.5}) is independent of the chosen origin.

\begin{proof} According to Theorem~\ref{thm 2.0} it suffices to
show that (\ref{1.20}) has a non-negative, non-identically zero
global sub-solution. This  will be obtained by joining suitable
radial local sub-solutions.
We set
$$\min_{\bdr B_r} a(x) = A(r) = A_+(r) - A_-(r),\quad
B(r) = \max_{\bdr B_r} b(x) ,\quad\text{and}\quad \tau = \max_{B_r}r\Delta r.
$$

Applying the previous lemma with $A_-(r)$ and $B(r)+\epsilon$
instead of $A(r)$ and $B(r),$ respectively, we deduce that
for every $\alpha_o$ and $\epsilon>0$ there exists a solution
$\alpha\in C^2([R, R+T_o])$
of differential inequality
\begin{equation}
\label{diff ineq alpha}
\alpha'' + (m-1) \frac{g'} g \alpha ' - A_- (r) \alpha - B(r)
\alpha^\sigma \geq 0,
\end{equation}
satisfying
$\alpha(R) = \alpha_o$,  $\alpha(R+T_o) = 0$, $\alpha'(r)< 0$
on $[R, R+T_o)$ and
\begin{equation}
\label{alpha' condition}
0>
\alpha' (R)> - \left( \frac
{g(R+T_o)}{g(R)}\right)^{m-1} \left\{ \frac 1{T_o} + T_o
\max_{[R,R+T_o]}[A(s) + (B(s)+\epsilon) \alpha_o^{\sigma
-1}]\right\}\alpha_o.
\end{equation}
Using the expression of the Laplacian of a radial function, and
the inequalities (\ref{1.5}),  $a(x) \geq - A_- (r(x))$ and $b(x)\leq B(r(x))$
and $\alpha'\leq 0$, it follows that the
function $v(x)=\alpha(r(x))$ is Lipschitz  in $B_{R+T_o}\setminus
B_R$ and $C^2$ in the complement of the cut locus of $o$ where it
satisfies the pointwise inequality
\begin{equation}
\label{subsol annulus}
\Delta v + a(x) v - b(x) v^\sigma =
\alpha'' +  \alpha'\Delta r + a(x) \alpha - b(x) \alpha^\sigma
\geq 0.
\end{equation}

To construct a sub-solution in the ball $B_R$ we   consider
the problem
\begin{equation}
\label{problem beta}
\begin{cases}
\beta'' + \ds{\frac{\tau}r} \beta' + A (r) \beta - B(r)
\beta^\sigma \geq 0 \quad  \text{on } [0,R]& \\
\beta>  0, \, \beta'\leq 0, \,\beta(R) = \alpha_o,\, \beta'(R)
\leq \alpha'(R),&
\end{cases}
\end{equation}
and we look for a solution  of the form
\begin{equation}
\label{1.25}
\beta(r) = \alpha_o \bigl( 1+ (R^2 - r^2)\eta \bigr)\geq \alpha_o.
\end{equation}
Note that $\beta'(r) = - 2\eta R\alpha_o<  0$, so, using (\ref{alpha'
condition}),  the condition $\beta' (R)\leq \alpha'(R)$  follows
from
\begin{equation}
\label{derivative estimate}
2\eta \geq \frac 1 R
\left( \frac {g(R+T_o)}{g(R)}\right)^{m-1} \left\{ \frac 1{T_o} + T_o
\max_{[R,R+T_o]}[A(s) + (B(s)+\epsilon) \alpha_o^{\sigma
-1}]\right\}.
\end{equation}
On the other hand, a direct computation that uses $A(r) \geq \min_{B_R}
a(x)$ and  $B(r)\leq \max_{B_R} b(x)$
shows that
\begin{equation*}
\beta'' + \frac{\tau}r \beta' + A (r) \beta - B(r)
\beta^\sigma \geq  \alpha_o \bigl\{\min_{B_R} a (x)
- 2\eta (1+\tau) - \alpha_o^{\sigma -1}
\eta^\sigma R^\sigma \max_{B_r} b(x)\bigr\}
\end{equation*}
and the right hand side is non-negative provided
\begin{equation}
\label{subsol condition}
2\eta \leq \frac  1 {1+\tau} \bigl\{ \min_{B_R} a -
\alpha_o^{\sigma -1} \eta^\sigma  R^\sigma \max_{B_r} b(x)
\bigr\}.
\end{equation}
Using (\ref{1.19}) we may choose $\alpha_o$ small enough that, for
every $\eta\leq  (\min_{B_R} a )/ [2(1+\tau)] $ (which is the largest
possible value of $\eta$ allowed by (\ref{subsol condition}))
the right hand side of (\ref{subsol condition}) is greater than or
equal to the right hand side of (\ref{derivative estimate}),
and therefore choose $\eta$ in such a way that both (\ref{derivative estimate}) and
(\ref{subsol condition}) are satisfied. For such values of $\alpha_o$ and
$\eta$,
the function $\beta$ satisfies all the requirements.
Proceeding as above one verifies that the function
$w(x) = \beta(\alpha(r(x))$ is Lipschitz in $B_R$ and $C^2$ in the complement
of the cut locus where it satisfies the pointwise inequality
\begin{equation}
\label{subsol ball}
\Delta w + a(x) w - b(x) w^\sigma
 \geq 0.
\end{equation}

Now we define
\begin{equation*}
u_-(x) =
\begin{cases}
w &\text{on }\, B_R\\
v &\text{on }\, B_{R+T_o} \setminus B_R \\
0 &\text{on }\, M\setminus B_{R+T_o},
\end{cases}
\end{equation*}
and claim that $u_-\in C^0(M) \cap H^{1}_{loc}(M)$ is a weak global
sub-solution of (\ref{1.20}).

This is easily seen if we assume that $o$ is a pole of $M$, for
then,  given a positive test function $\vp\in C^\ty_c(M)$,  applying
Green's second identity, and using the fact  that
$w$ and $v$ are  pointwise sub-solutions of (\ref{1.20}) in $B_R$ and
$B_{R+T_o}\setminus \overline B_R $ respectively,  we obtain
\begin{multline*}
\int_M u_- \Delta \vp = \int_M \langle \nabla u_- , \nabla \vp\rangle
= \int_{B_R} \vp \Delta w + \int_{B_{R+T_o}\setminus B_R} \vp \Delta v
\\
+
\int _{\bdr B_R} (w-v) \langle \nabla \vp ,
\nabla r\rangle - \vp \langle \nabla (w-v) , \nabla r\rangle
+\int
_{\bdr B_{R+T_o}} v \langle \nabla \vp ,
\nabla r\rangle - \vp \langle \nabla v , \nabla r\rangle
\\
\geq \int_M \bigl( -a(x)u_- + b(x) u_-^\sigma \bigr) \vp -
\bigl(\beta'(R) - \alpha'(R)\bigr) \int _{\bdr B_R} \vp
-\alpha'(R+T_o) \int_{\bdr B_{R+T_o}} \vp,
\end{multline*}
and the claim follows from the inequalities $\beta'(R)\leq
\alpha'(R)$, $\alpha'(R+T_o)\leq 0.$

In the case where the cut locus
of $o$ is not empty, one can adapt an argument in \cite{PRS2}, Lemma 2.2,
as follows:
we consider an exhaustion $\O_n$ of $M\setminus Cut_o$ by domains with smooth boundary,
which are star shaped with respect to $o$, so that, denoting by $\nu$ the
outward unit normal, we have $\langle \nabla r , \nu\rangle > 0$ on
$\bdr \Omega_n$. Since the part of $\bdr B_R $ contained in $\Omega_n$
is smooth, we may also deform $\Omega_n$, if necessary,  in such a way
that , for every $n$, $\bdr  \Omega_n$ is transversal to $\bdr B_R$ and $\bdr B_{R+T_o}$.

Since $\nabla u_-$ is locally bounded we have
\begin{equation*}
\int_{M} u_-\Delta \vp = -\int_M \langle\nabla u_-, \nabla
\vp\rangle =- \lim_{n\to \ty} \int_{\Omega_n}\langle \nabla u_- , \nabla \vp
\rangle.
\end{equation*}
We write $\Omega_n = (\Omega_n \cap B_R) \cup (\Omega_n\cap (B_{R+T_o}\setminus B_R))
\cup(\Omega_n \cap  B_{R+T_o}^c),$
apply the divergence theorem,  use the fact that $v$ and $w$ are
pointwise sub-solutions of (\ref{1.20}) in the complement of the cut
locus, to obtain
\begin{multline*}
-\int_{\Omega_n}
\langle \nabla u_- , \nabla \vp
\rangle
\geq
\int_{\Omega_n} \vp
\bigl[ -a(x) u_- + b(x) u_-^\sigma\bigr]
\\
-\Bigl( \int_{\bdr (\Omega_n \cap B_R)} \vp \langle \nabla w, \nu\rangle
+ \int_{\bdr (\Omega_n \cap (B_{R+T_o}\setminus B_R)} \vp \langle \nabla v,
\nu\rangle\Bigr).
\end{multline*}
To conclude, note that, by the transversality assumption,
up to sets of lower dimension, we have $\bdr \bigl( \Omega_n \cap B_R\bigr)
= (\Omega_n\cap \bdr B_R )\cup (\bdr \Omega_n \cap B_R) $
and similarly when $B_R $ is replaced by $B_{R+T_o}\setminus B_R,$ so that the
boundary integrals become
\begin{multline*}
(\beta'(R) - \alpha'(R)) \int_{\Omega_n\cap \bdr B_R} \vp +
\alpha'(R) \int_{\Omega_n\cap \bdr B_{R+T_o}} \vp \\
+\int_{\bdr \Omega_n \cap B_R}  \vp  \beta'\langle \nabla r, \nu\rangle
+\int_{\bdr \Omega_n \cap (B_{R+T_o}\setminus B_R)}  \vp \alpha' \langle \nabla r, \nu\rangle
\end{multline*}
and the first integral is non-positive because $\beta'(R)\leq \alpha'(R)$,
while the last two are non-positive because $\alpha',$ $\beta' \leq 0$ and
$\langle\nabla r, \nu\rangle\geq 0 $ on $\bdr \Omega_n.$
\end{proof}

\begin{remark}
\label{Yamabe}{\rm
The above existence result is also relevant  to the Yamabe problem,
that is, the possibility of conformally deforming the assigned
metric, with scalar curvature $s(x)$, to a new one with prescribed scalar
curvature $K(x)$. Indeed assume that $m\geq 3,$ and denote by  $u^{2/(m-2)}$
 the conformal factor, so that the deformed metric is given by
$\widetilde{\langle \,,\,\rangle} =  u^{\frac 2{(m-2)}}
\langle \,,\,\rangle$.
Then the scalar curvature of $\widetilde{\langle \,,\,\rangle}$ is $K(x)$
provided $u$ is a,  necessarily positive, solution of the Yamabe
equation
\begin{equation*}
c_m \Delta u -s(x) u + K(x) u^{\frac{m+2}{m-2}} =0
\end{equation*}
where $c_m= \frac{4(m-1)}{m-2},$ which  is (\ref{1.20}) with $a(x) =
- c_m^{-1}s(x)$ and $b(x) = - c_m^{-1} K(x)$.

By way of example, assume that $\mathrm{Ricc} \geq -(m-1) B$, so
that $s(x) \geq - m(m-1)B$, and, as noted above, (\ref{1.5}) holds
with $g(r)= (\sqrt B)^{-1} \sinh (\sqrt B r)$. If we suppose that
$s(x)\leq 0$ in a ball of radius $R+1$, then (\ref{1.19}) holds,
with $T_o=1$, provided
\begin{equation*}
R  \sup_{B_R} s(x) \leq - c_m \left(
 1 + (m-1)R \sqrt B \coth (\sqrt B R)
 \right)
 \left(\frac{\sinh (\sqrt B (R+1))}{\sinh (\sqrt B
 R)}\right)^{m-1}.
\end{equation*}
It is easy to see that if $\sqrt B > 4 (m-1)/[m(m-2)]$ then
the above  condition is verified provided
$s(x)$ is sufficiently near $-m(m-1)B$ in the ball $B_R$ and $R$
is large enough. We note that there are situations (for instance  when
the sectional curvature is  suitably pinched) where
Theorem~\ref{thm 2.1} is applicable, while Theorem~\ref{thm BRS} is not.
}
\end{remark}

\section{Non existence results}
The purpose of this section is to prove a non-existence result for
non-negative $C^2$ solutions of the differential inequality
\begin{equation}
\label{3.1}
\Delta u + a(x) u - b(x) u^\sigma \geq 0 \qquad \text{on }\, M.
\end{equation}

We begin with the following general

\begin{theorem}
\label{thm 3.2} Let $a(x)$, $b(x)\in C^0(M)$ and assume that $b(x)\geq 0.$
Let $H>0$, $K>-1$ and $A\in \R$ be constants satisfying
\begin{equation}
\label{3.3}
\max \{ 0 , A \} \leq H(K+1) -1,
\end{equation}
and suppose that there exists a positive $C^2$ solution of the
differential inequality
\begin{equation}
\label{3.4}
\Delta \vp + H a(x) \vp \leq -K \frac{|\nabla \vp|^2}{\vp} \qquad
\text{on }\,\, M.
\end{equation}
Then the differential inequality
\begin{equation}
\label{3.5}
u\Delta u + a(x) u^2 - b(x) u^{\sigma +1} \geq -A |\nabla u|^2,
\qquad\sigma \geq 1,
\end{equation}
has no non-negative $C^2$ solutions on $M$ satisfying
\begin{equation}
\label{3.6}
\mathrm{supp }\,u \cap \{x\in M\,:\, b(x)>0\}\neq \emptyset
\end{equation}
and
\begin{equation}
\label{3.7}
\Bigl(\int_{\bdr B_r} \vp^{\frac{\beta+1} H(2-p)} u^{2(\beta +1)}\Bigr)^{-1}\not\in L^1(+\ty),
\end{equation}
for some $p>1$ and $\beta$ satisfying $\max\{0,A\} \leq\beta \leq H(K+1)-1.$
\end{theorem}

\medskip
\noindent
{\bf Remarks.}
We note that if $p=2$, then the non-integrability assumption
(\ref{3.7}) involves $u$ alone and reduces to
\begin{equation}
\label{3.7bis}
\Bigl(\int_{\bdr B_r}
u^{2(\beta +1)}\Bigr)^{-1}\not\in L^1(+\ty).
\end{equation}
If $p\ne 2$, assumption (\ref{3.7bis}) implies (\ref{3.7}) if a suitable bound
on $\vp$ is available, e.g., if $p<2$ and $\vp$ is bounded from above.
A similar simplification occurs if $\vp$ is bounded above, respectively below, by
a  radial function, see, e.g., Theorem~\ref{thm 3.2'} below.

We also remark   that the inequality
\begin{equation*}
r-R = \int_R^r  f^{1/2} f^{-1/2} \leq \bigl(\int_R^r f\bigr )^{1/2}
\bigl(\int_R^r f^{-1}\bigr )^{1/2}
\end{equation*}
valid for $f>0$, together with integration in polar coordinates,  shows that
condition (\ref{3.7}) is implied by $ \vp^{\frac{2-p} H} u\in L^{2(\beta +1)}(M)$.

Finally, we note that the proof of the theorem could be achieved by
adapting the argument in the proof of Theorem~1.3 in \cite{PRS3}.
However, the present assumptions allow us to give an alternative
argument that we describe here for the sake of completeness.
\par

\begin{proof}
Let $u\geq 0$ be a solution of (\ref{3.5}) on $M$ satisfying
(\ref{3.6}) and (\ref{3.7}). Fix $\epsilon >0$, $\alpha\in \R$
and set
\begin{equation*}
v = \vp^{-\alpha} (u^2 +\epsilon)^{(\beta+1)/2}.
\end{equation*}
A straightforward computation that uses (\ref{3.4}) and
(\ref{3.5}) yields
\begin{equation*}
\begin{split}
v\div{\vp^{2\alpha} \nabla v}
& \geq
\alpha (K-\alpha +1) (u^2+\epsilon)^{\beta+1}
\frac{|\nabla \vp|^2}{\vp^2}
+ (\beta + 1) (u^2 + \epsilon)^\beta b(x) u^{\sigma +1}
\\
& + a(x) (u^2 +\epsilon)^{\beta+1}
\bigl[\alpha H - (\beta+1) \frac{u^2}{u^2+\epsilon}\bigr]\\
& + (\beta +1 ) (u^2 +\epsilon)^{\beta}
\bigl[
1-A + (\beta - 1) \frac{u^2}{u^2+\epsilon}
\bigr] |\nabla u|^2.
\end{split}
\end{equation*}
We choose $\alpha= H^{-1} (\beta+1)$, so that our assumptions on
$\beta$, $H$ and $K$ yield $0<\alpha\leq K+1$. Therefore
$\alpha(K-\alpha +1)\geq 0,$ and using the assumptions $b(x)\geq
0$ and $\beta+1\geq 0,$ we deduce that
\begin{equation}
\label{3.8}
\begin{split}
v\div{\vp^{2\alpha} \nabla v}
& \geq \epsilon (\beta+1) a(x)
(u^2+\epsilon)^\beta\\
&+ (\beta+1) (u^2+\epsilon)^\beta
\bigl[
1-A + (\beta - 1) \frac{u^2}{u^2+\epsilon}
\bigr] |\nabla u|^2.
\end{split}
\end{equation}
Let $r(t)\in C^1(\R)$ and $s(t)\in C^0(\R)$ satisfy the conditions
\begin{equation}
\label{3.9}
\rmi\,\, r(v)\geq 0, \qquad r(v)+vr'(v)\geq s(v)>0, \quad \text{on
}\, [0,+\ty),
\end{equation}
and let $Z$ be the vector field defined by
$Z= v r(v)\vp^{2\alpha} \nabla v$. For fixed
$t$ and $\delta>0$  let also  $\psi_\delta$ be the
Lipschitz function defined by
\begin{equation*}
\psi_\delta(x) =
\begin{cases}
1 & \text{if }\, r(x)\leq t\\
\ds{\frac{t+\delta-r(x)}\delta} &\text{if } \, t<r(x)< t+\delta \\
0 &\text{if } \, r(x)\geq t+\delta.
\end{cases}
\end{equation*}
Using (\ref{3.8}) (\ref{3.9}) and the definition of $\psi_\delta$
we compute
\begin{equation*}
\begin{split}
\div{\psi_\delta Z}  &= \psi_\delta \diver Z + \langle \nabla
\psi_\delta, Z\rangle\\
&\geq
(\beta+1)r(v) (u^2+\epsilon)^\beta \bigl(\epsilon a(x) +
\bigl[1-A + (\beta-1) \frac{u^2}{u^2+\epsilon}\bigr] |\nabla u|^2
\bigr)\chi_{B_{t}}  \\
& +  s(v)\vp^{2\alpha} |\nabla v|^2\chi_{B_{t}}
+ \frac 1\delta \langle \nabla r , Z\rangle \chi_{\overline B_{t+\delta}
\setminus B(t)},
\end{split}
\end{equation*}
whence, integrating, and using the divergence theorem and the
Cauchy-Schwarz inequality we obtain
\begin{multline*}
\int_{B_t}(\beta+1)r(v) (u^2+\epsilon)^\beta \bigl(\epsilon a(x) +
\bigl[1-A + (\beta-1) \frac{u^2}{u^2+\epsilon}\bigr] |\nabla u|^2
\bigr)\\
+ \int_{B_t}  s(v)\vp^{2\alpha} |\nabla v|^2
\leq
\frac 1\delta \int_{\overline B_{t+\delta} \setminus B(t)}
| Z|.
\end{multline*}
By H\"older inequality the integral on the right
hand side is bounded above by
\begin{equation*}
\Bigl(\frac 1\delta \int_{\overline B_{t+\delta} \setminus B(t)}
\vp^{2\alpha} \frac{r(v)^2}{s(v)} v^2\Bigr)^{1/2}
\Bigl(\frac 1\delta \int_{\overline B_{t+\delta} \setminus B(t)}
\vp^{2\alpha} s(v) |\nabla v| ^2\Bigr)^{1/2}.
\end{equation*}
Inserting into the above inequality,  letting $\delta\to 0+$ and using the
co-area formula (see Theorem 3.2.12  in \cite{F}) we deduce that
\begin{multline}
\label{3.12}
\int_{B_t}(\beta+1)r(v) (u^2+\epsilon)^\beta \bigl(\epsilon a(x) +
\bigl[1-A + (\beta-1) \frac{u^2}{u^2+\epsilon}\bigr] |\nabla u|^2
\bigr)  +  s(v)\vp^{2\alpha} |\nabla v|^2
\\
\leq
\Bigl( \int_{\bdr B_{t}} \vp^{2\alpha}
\frac{r(v)^2}{s(v)} v^2\Bigr)^{1/2}
\Bigl( \int_{\bdr B_{t} } \vp^{2\alpha}
s(v) |\nabla v| ^2\Bigr)^{1/2}.
\end{multline}
In the above formula, the surface integral is computed with
respect to $(m-1)$-dimensional Hausdorff measure on $\bdr B_t$,
which coincides with the Riemannian measure induced on the regular part
of $\bdr B_t$ (the intersection of $\bdr B_t$ with the complement of the
cut locus of $o$, see \cite{F}, 3.2.46,
or \cite{Ch}, Proposition 3.4). As $\e \to 0,$ $v=v_\e\to
v_0= \vp^{-\alpha} u^{\beta+1}$, whence, using the dominated convergence
theorem in (\ref{3.12}), we get
\begin{multline}
\label{3.12'}
\int_{B_t}  s(v_0)\vp^{2\alpha} |\nabla v_0|^2 +
(\beta+1)(\beta -A ) \int_{B_t}r(v_0) u^{2\beta}  |\nabla u|^2
\bigr)
\\
\leq
\Bigl( \int_{\bdr B_{t}} \vp^{2\alpha}
\frac{r(v_0)^2}{s(v_0)} v_0^2\Bigr)^{1/2}
\Bigl( \int_{\bdr B_{t} } \vp^{2\alpha}
s(v_0) |\nabla v_0| ^2\Bigr)^{1/2}.
\end{multline}
Defining
\begin{equation*}
h(t) =\int_{ B_{t} } \vp^{2\alpha}
s(v_0) |\nabla v_0| ^2,
\end{equation*}
so that by the co-area formula $H$ is Lipschitz and
\begin{equation*}
h'(t) =\int_{\bdr B_{t} } \vp^{2\alpha}
s(v_0) |\nabla v_0| ^2,
\end{equation*}
and noting that the coefficient of the second integral on the
left hand side of (\ref{3.12'}) is non-negative
by the conditions imposed on $\beta,$ we obtain
\begin{equation}
\label{3.13}
h(t)\leq
\Bigl( \int_{\bdr B_{t}} \vp^{2\alpha}
\frac{r(v_0)^2}{s(v_0)} v_0^2\Bigr)^{1/2}
\bigl( h'(t) \bigr)^{1/2}
\end{equation}
Our aim is to show that under assumption (\ref{3.7}) $v_0$ is
constant. The proof follows the lines of that of Lemma~1.1 in
\cite{RS}. Assume by contradiction that $v_0$ is not constant.
Then there exists $R_o$ such that $h(t)>0$ for every $t\geq R_o$,
and therefore the right hand side of (\ref{3.13}) is positive for
$t\geq R_o$. Dividing through by $h(t)$, squaring and integrating
the resulting differential inequality between $R$ and $r$ with
$R_o\leq  R< r$ yield
\begin{equation}
\label{3.14}
h(R)^{-1}\geq h(R)^{-1} - h(r)^{-1} \geq \int_R^r
\Bigl(\int_{\bdr B_t} \vp^{2\alpha}
\frac{r(v_0)^2}{s(v_0)} v_0^2 \Bigr)^{-1} dt.
\end{equation}
We choose a sequence of functions
\begin{equation*}
r_n(t) = (t^{2}+\frac{1}{n})^{\frac{p-2}{2}
}, \quad  s_{n}(t)= \min\{p-1, 1\}\, r_n(t)\,\,
\text{, }\forall n\in\mathbb{N}, \,\, p>1.%
\end{equation*}
Since  condition (\ref{3.9}) holds for every $n,$ so does (\ref{3.14}),
whence, letting $n\to +\ty$ and  using the Lebesgue and monotone convergence
theorems we deduce that there exists $C>0$ which depends only on $p$ such
that
\begin{equation}
\label{3.16}
\Bigl( \int_{B_R} v_0^{p-2} \vp^{2\alpha} |\nabla v_0|^2\Bigr)^{-1}
\geq C \int_R^r
\Bigl(\int_{\bdr B_t} \vp^{2\alpha} v_0^p \Bigr)^{-1} dt.
\end{equation}
Now, recalling that $\alpha= (\beta+1)/H$ and the definition of $v_0$,
we have  $\vp^{2\alpha}v_0^p = \vp^{(2-p)(\beta+1)/H}
u^{2(\beta+1)}$
and
the required contradiction is reached by letting $r\to +\ty$
and using assumption (\ref{3.7}).

Thus $v_0$ is constant, and we deduce that there exists a constant
$C\geq 0$ such that
\begin{equation*}
u^H=C \vp.
\end{equation*}
Since $u$ is not identically zero by (\ref{3.6}), $C>0$ and $u$ is
strictly positive on $M.$ We insert the expression of $\vp$ in
terms of $u$ in (\ref{3.4}), divide by $CH u^{H-2}$ and subtract the
result from (\ref{3.5}) to obtain
\begin{equation*}
[A-H(K+1) + 1] |\nabla u|^2 \geq b(x) u^{\sigma +1}.
\end{equation*}
Since the coefficient of $|\nabla u|^2$ is  non-positive, by
(\ref{3.3}), we conclude that
\begin{equation*}
b(x) u^{\sigma +1} \leq 0,
\end{equation*}
which contradicts (\ref{3.6}).
\end{proof}

\begin{remark}
\label{rmk 3.16-1}
{\rm
Observe that the above proof actually shows that if $A< H(K+1)-1$
then $\nabla u= 0$, so that   $u$, and therefore $\vp$, are
necessarily constant.
It follows from (\ref{3.4}) and (\ref{3.5}) that
$0\geq a(x)u \geq b(x) u^{\sigma +1} \geq 0$, so that,
if $a$ does not vanish identically, then $u\equiv 0$, without any
assumption on $b$.  On the other hand, if  $A=H(K+1)-1,$
then the conclusion depends on the fact that $b$ is
positive somewhere.
}
\end{remark}

\begin{remark}
\label{rmk 3.16-0}
{\rm
We also note that if $u$ is assumed to be strictly positive, then the
conclusion of the Theorem holds, with a much easier proof, if we assume that
$\max\{-1, A\}  \leq H(K+1)-1,$ and that $\beta>-1$, $\beta\geq
A,$ $\beta\leq H(K+1)-1.$
}
\end{remark}

\begin{remark}
\label{rmk 3.16a}
{\rm
Let $L_H$ be the Schr\"odinger operator defined by $L_H= \Delta +H
a(x).$ Then the validity of (\ref{3.4}) is related to the sign of
$\lambda_1^{L_H}(M)$. Indeed, assume that $\vp$ is a positive
$C^2$ solution of (\ref{3.4}). Let $\psi\in C_c^{\ty}(M)$ and
apply the divergence theorem to the vector field
$\psi^2 \nabla \log \vp$.
Since, by (\ref{3.4}) and Young inequality
\begin{equation*}
\begin{split}
\div{\psi^2 \nabla \log \vp} &= \frac{\psi^2}\vp \bigl(\nabla \vp
- \frac{|\nabla \vp|^2}\vp\bigr) +2\frac \psi\vp \langle \nabla
\vp,\nabla \psi\rangle \\
&\leq
\frac{\psi^2}\vp
\biggl(
-H a(x) \vp -(K+1)\frac{|\nabla \vp|^2}\vp
\biggr) +
\frac{\psi^2}{\vp^2} |\nabla \vp|^2 + |\nabla \psi|^2,
\end{split}
\end{equation*}
we deduce that
\begin{equation*}
\int_M |\nabla \psi|^2 - H a(x) \psi^2 \geq K \int_M \psi^2
\frac{|\nabla \vp|^2}{\vp^2},
\end{equation*}
and from the variational characterization of the bottom of the spectrum we
we conclude that  if $K\geq 0$ then $\lambda_1^{L_H}(M)\geq 0.$

On the other hand, if $\lambda_1^{L_H}(M)\geq 0$, then, by  an
extension of the result of  Moss Pieperbrink, and Fisher-Colbrie
Schoen quoted in Section~1 (see  \cite{PRS3}, Lemma~1.2), there exists
a positive $C^1$ function $v$ which satisfies
\begin{equation*}
\Delta v + H a(x) v =0
\end{equation*}
weakly on $M.$ Further, if $a(x)$ is assumed to be $C^{0,\alpha}$ for some
$\alpha\in(0,1),$  then $v$ is $C^2$ and it is a classical solution of the
above equation. It is clear that $v$ is  respectively a weak or a classical,
solution of (\ref{3.4}) for every $K\leq 0.$
}
\end{remark}

\begin{corollary}
\label{cor 3.17}
Let $a(x),$ $b(x)\in C^0(M)$ and assume that $b(x)$ is
non-negative and does not vanish  identically. Suppose also that,
for some $H\geq 1$, $\lambda_1^{L_H}(M)\geq 0$. Then there are no
positive $C^2$ solutions of the differential inequality
\begin{equation}
\label{3.20b}
\Delta u + a(x) u - b(x) u^\sigma \geq 0, \qquad \sigma \in \R
\end{equation}
such that
\begin{equation}
\label{3.7b}
\Bigl(\int_{B_r} u^{2(\beta +1)} \Bigr)^{-1}\not\in L^1(+\ty)
\end{equation}
for some $0\leq\beta\leq H-1.$
\end{corollary}

If $a(x)$ is assumed to be $C^{0,\alpha}$, the the corollary
follows immediately from Remark~\ref{rmk 3.16a} and from
Theorem~\ref{thm 3.2} with $K=0$. In the general case, the function
$v$ satisfies inequality (\ref{3.4}) with $K=0$ only in weak sense, and
the argument of Theorem~\ref{thm 3.2} needs to be slightly modified to be
carried  out in this situation. Note also that, since the corollary deals
with  strictly positive solutions, we can drop the assumption that
$\sigma\geq - 1$.

In order to apply Theorem~\ref{thm 3.2} and obtain the non-existence result
mentioned at the beginning of this section, one needs to verify that the
(non-)integrability condition (\ref{3.7}) holds. In principle, this may be
obtained combining a-priori upper estimates for $u$ with appropriate bounds
for the volume growth of balls. Both estimates can be deduced imposing
lower bounds on the radial Ricci curvature of the manifold.

In the following lemma we deduce an a-priori integral estimate for
nonnegative solutions of (\ref{3.5}), which will enable us to
obtain (\ref{3.7}) under the sole assumption of a volume growth
condition.

\begin{lemma}
\label{lemma 3.1}
Let $(M,\langle \,,\,\rangle)$ be a complete Riemannian manifold,
and let $a(x), b(x)\in C^0(M)$ with $b(x)>0$ on $M.$ Assume
that $u\geq 0$ is a $C^2$ solution of the differential inequality
\begin{equation}
\label{3.a} u\Delta u +a(x) u^2 - b(x) u^{\sigma +1} \geq -A
|\nabla u|^2,
\end{equation}
for $A\leq 1$ and $\sigma>1.$ Then for every $p\geq 1,$ $p>A+2$
there exist constants $C_1,$ $C_2>0$ which depend only on
$p$, $\sigma$ and $R_0>0$  such that, for every $R\geq R_0,$
\begin{equation}
\label{3.aa}
\int_{B_R} b(x) u^{p+\sigma -2} \leq
{C_1}{R^{-2\frac{p+\sigma-2}{\sigma-1}}} \int_{B_{2R}}
b(x)^{-\frac{p-1}{\sigma-1}} +
C_2 \int_{B_{2R}}
\Bigl(\frac{a_+(x)}{b(x)}\Bigr)^{\frac{p-1}{\sigma-1}}a_+(x)
\end{equation}
\end{lemma}

\begin{proof}
Observe first that  we may assume that $u\not\equiv 0$, for otherwise
there is nothing to prove. Thus, there exists $R_0>0$ such that
$u\not\equiv 0$ on $B_{R}$  for every $R\geq R_0.$

Next, for  every $R\geq R_0,$ let $\psi=\psi_R:M\to[0,1]$ be a
smooth cut-off function such that
\begin{equation}
\label{3.b}
\psi\equiv 1 \,\, \text{ on }\, B_R,\quad
\psi\equiv 0 \,\,\text{ on }\, M\setminus B_{2R},
\,\,\text{and}\,\,
|\nabla \psi| \leq \frac CR \psi^{\frac{p-1}{p+\sigma -2}} \,\,
\text{ on }\, B_{2R},
\end{equation}
for some $C$ which depends only on $p$ and $\sigma$. Note that this
is possible since the exponent $\frac{p-1}{p+\sigma -2}$ is
strictly less than $1.$ Having fixed $\epsilon >0,$ we let $W$ be the
vector field defined by
\begin{equation*}
W=\psi^2 (u+\epsilon)^{p-3} u\nabla u.
\end{equation*}
A computation that uses (\ref{3.a})
yields
\begin{equation*}
\begin{split}
\diver W &\geq
\psi^2(u+\epsilon)^{p-3}
\bigl\{
-a(x) u^2 + b(x)u^{\sigma +1} +\bigl(1-A -(p-3)\frac
u{u+\epsilon}\bigr)|\nabla u|^2\bigr \}
\\
&+2\psi (u+\epsilon)^{p-3}u
\langle \nabla u, \nabla \psi\rangle.
\end{split}
\end{equation*}
We estimate the last term on the right hand side using
Cauchy-Schwarz's inequality and Young's inequality
$2ab\leq \lambda a^2 + \lambda^{-1} b^2$ with $\lambda = p-2-A>0$,
to obtain
\begin{equation*}
\diver W  \geq
\psi^2(u+\epsilon)^{p-3}
\bigl\{
-a_+(x) u^2 + b(x)u^{\sigma +1}
\bigl\} - \frac 1{p-2-A}u(u+\epsilon)^{p-2} |\nabla \psi|^2.
\end{equation*}
We integrate the above inequality, apply  the divergence theorem,
rearrange, let $\epsilon \to 0+$ and use the dominated convergence
theorem, in this order, to deduce that
\begin{equation}
\label{3.d}
\int_{B_{2R}} b(x)\psi^2 u^{p+\sigma -2} \leq
\frac 1{p-2-A}\int_{B_{2R}}  u^{p-1} |\nabla \psi|^2
+
\int_{B_{2R}} \psi^2 a_+(x) u^{p-1}
\end{equation}
If $p=1$ the conclusion follows immediately using (\ref{3.b}).
If $p>1,$, we denote  by $I$ and $II$ the two integrals on the right
hand side, and use H\"older inequality with conjugate exponents
\begin{equation*}
\frac{p+\sigma -2}{p-1}(>1) \quad\text{and}\quad \frac{p+\sigma
-2}{\sigma-1},
\end{equation*}
and the assumption that $b(x)>0$ to estimate
\begin{equation*}
I\leq \Bigl(\int_{B_{2R}} b(x) \psi^2 u^{p+\sigma
-2}\Bigr)^{\frac{p-1}{p+\sigma -2}}
\Bigr(
\int_{B_{2R}}
\psi^{-2 \frac{p-1}{\sigma-1}} b(x)^{-\frac{p-1}{\sigma-1}}
|\nabla\psi|^{2\frac{p+\sigma -2}{\sigma-1}}
\Bigr)^ {\frac{\sigma -1}{p + \sigma-2}}
\end{equation*}
and
\begin{equation*}
II\leq
\Bigl(\int_{B_{2R}} b(x) \psi^2 u^{p+\sigma
-2}\Bigr)^{\frac{p-1}{p+\sigma -2}}
\Bigr(
\int_{B_{2R}} \psi^2 a_+(x)^{\frac{p+\sigma
-2}{\sigma-1}}b(x)^{-\frac{p-1}{\sigma - 1}}
\Bigr)^ {\frac{\sigma -1}{p + \sigma-2}}
\end{equation*}
Inserting into (\ref{3.d}), noting that the integral on the
left hand side is strictly positive by the choice of $R$, and simplifying,
we obtain
\begin{multline*}
\int_{B_{2R}} b(x)\psi^2 u^{p+\sigma -2}
\leq
\Bigr\{
\frac 1{p-2-A}\Bigl(\int_{B_{2R}}
\psi^{-2 \frac{p-1}{\sigma-1}} b(x)^{-\frac{p-1}{\sigma-1}}
|\nabla\psi|^{2\frac{p+\sigma -2}{\sigma-1}}
\Bigr)^ {\frac{\sigma -1}{p + \sigma-2}}
\\
+
\Bigl(
\int_{B_{2R}} \psi^2 a_+(x)^{\frac{p+\sigma
-2}{\sigma-1}}b(x)^{-\frac{p-1}{\sigma - 1}}
\Bigr)^{\frac{\sigma -1}{p + \sigma-2}}
\Bigr\}^{\frac{p+\sigma -2}{\sigma -1}}
.
\end{multline*}
The required conclusion  follows again using (\ref{3.b}) and the
elementary inequality $(a+b)^\tau \leq 2^\tau(a^\tau +b^\tau)$ valid for
$a,b, \tau\geq 0$.
\end{proof}

We are now ready for our main non-existence result.

\begin{theorem}
\label{thm 3.3}
Let $(M,\langle \,,\,\rangle)$ be a complete Riemannian manifold,
and let $a(x),$ $ b(x)$$\in C^0(M)$ where  $b(x)>0$ on $M$ and
\begin{equation}
\label{a}
b(x)\geq \frac C{r(x)^{\mu}}
\end{equation}
for $r(x)\gg 1$ and for some constants $C>0$ and $0\leq \mu \leq 2.$
Assume that
\begin{equation}
\label{b}
\rmi \, \sup_M \frac{a_+(x)}{b(x)}<+\ty \,\, \text{and}\,\,
\rmii\,
\int_{B_r} a_+(x) = O\bigl( r^{2-\mu} \log r \bigr) \, \text{ as
}  r\to +\ty,
\end{equation}
and that, for some $H\geq 1$, the operator $L_H = \Delta +H a(x)$
satisfies
\begin{equation}
\label{c}
\lambda_1^{L_H}(M)\geq 0.
\end{equation}
Finally, let $A$ and $\sigma$ be such that $A\leq 1,$  $A<H-1$,
$1<\sigma \leq 2H+1$ and $\sigma <2H-A$ and assume that
\begin{equation}
\label{e}
\vol B_r = O\Bigl( r^{2 +(2-\mu) \frac {2H}{\sigma-1}}
\log r\Bigr)
\text{ as
} \, r\to +\ty.
\end{equation}
Then the only non-negative $C^2$ solution $u$ of the differential
inequality
\begin{equation}
\label{d}
u\Delta  u+ a(x) u^2 - b(x) u^{\sigma +1} \geq -A |\nabla u|^2
\end{equation}
is $u\equiv 0.$
\end{theorem}

\begin{proof}
If we set $p=2H +2-\sigma,$ the conditions imposed on the
parameters imply that $p$ satisfies the assumptions listed in the
statement of Lemma~\ref{lemma 3.1}. The lemma and condition (\ref{b}) \rmi
show  that  there exist constants $C_i>0$
such that
\begin{equation}
\label{f}
\int_{B_r} b(x) u^{2H} \leq
{C_1}{r^{-\frac{4H}{\sigma-1}}} \int_{B_{2r}}
b(x)^{1-\frac{2H}{\sigma-1}} +
C_2 \int_{B_{2r}}
a_+(x)
\end{equation}
for $r>0$ sufficiently large. We use condition (\ref{a}) to estimate
from below the integral on the left hand side. On the other hand, since
$\sigma <2H+1,$ we may again use condition (\ref{a}) to estimate from above
the first integral on the right hand side, and (\ref{b}) \rmii to estimate
from above the second integral,  and deduce that, for $r$ sufficiently large,
\begin{equation*}
\int_{B_r} u^{2H} \leq C
\bigl( r^{(\mu -2)\frac{2H}{\sigma -1}}\vol
B_{2r} + r^2\log r
\bigr),
\end{equation*}
whence, using the volume growth condition (\ref{f}) we conclude
that
\begin{equation*}
\int_{B_r} u^{2H} \leq Cr^2 \log r \qquad \text{for }\, r\gg 1.
\end{equation*}
This immediately implies that, for $r$ large,
\begin{equation*}
\frac{ r}{\int_{B_r} u^{2H}} \geq C \frac 1{r\log r} \not\in
L^1(+\ty),
\end{equation*}
which in turn yields (see, e.g., \cite{RS} Proposition~1.3)
\begin{equation*}
\frac 1{\int_{\bdr B_r} u^{2H} } \not \in L^1(+\ty).
\end{equation*}
We may therefore apply Theorem~\ref{thm 3.2} with $K=0$ and $\beta = H-1$
to deduce that $\mathrm{supp}\, u=\emptyset$, that is, $u\equiv 0.$
\end{proof}

\begin{remark}
\label{rmk a}
{\rm
The argument used in the proof shows that the condition that $\frac{a_+}{b}$ is
bounded above may be removed provided we replace (\ref{f}) with
\begin{equation}
\label{f'}
\int_{B_r} a_+^{\frac{2H}{\sigma -1}} = O
\Bigl(
r^{2-\mu\frac{2H}{\sigma -1}} \log r
\Bigr) \quad \text{as }\, r\to +\ty.
\end{equation}
Note that since the integral on the left hand side is a non-decreasing
function of $r$ this also imposes the further restriction
$\mu\leq (\sigma-1)/H$, with corresponding restrictions being imposed
on the range of the other parameters.
}
\end{remark}

\begin{remark}
\label{rmk b}
{\rm
In the case where  the ambient manifold is Euclidean space, we can
compare our Theorems~\ref{thm 3.2} and \ref{thm 3.3} with the
results in \cite{AB}, Section~3. We consider the equation
\begin{equation}
\label{0.2'}
\Delta u +\lambda a(x) u - u^2=0 \quad\text{on }\, \R^m,
\end{equation}
which, with  $a(x)=g(x)$ and a change of scaling,
is easily seen to be equivalent to (\ref{0.2}). We assume, as in
\cite{AB}, that $a(x)$ is positive somewhere, and that its positive
part $a_+(x)$ satisfies the estimate
\begin{equation*}
a(x)\leq \frac{k}{|x|^2},
\end{equation*}
for some positive constant $k.$ According to the discussion at
the end of Section~1, it follows that the principal eigenvalue
$\lambda_*$ of the linear equation associated to (\ref{0.2'})
is strictly positive and satisfies
\begin{equation*}
\lambda_*\geq \frac{(m-2)^2}{4k}.
\end{equation*}
Moreover, if $\lambda \leq \lambda_*$,  we have
$\lambda_1^{\Delta +H\lambda a(x)}(\R^m)\geq 0$ provided
$H\leq \frac{\lambda_*}{\lambda}$. On the other hand, if $u$ is a
non-negative solution of (\ref{0.2'}),  Lemma~\ref{lemma 3.1}
with $A=0,$ $\sigma=2$ and $p>2$ shows that
\begin{equation*}
\int_{B_r} u^p \leq C
\begin{cases}
r^{m-2p} &\text{if } \, m-2p >0\\
\log r &\text{if } \, m-2p =0\\
1 &\text{if } \, m-2p <0,
\end{cases}
\end{equation*}
and therefore
\begin{equation}
\label{not int}
\frac{r}{\int_{B_r} u^p} \not \in L^1(+\ty)
\end{equation}
provided
\begin{equation*}
p\geq \frac{m-2}2.
\end{equation*}
In order to apply Corollary~\ref{cor 3.17}, the non-integrability
condition must hold with $p$ satisfying
\begin{equation*}
p= 2 (\beta + 1)\leq 2H.
\end{equation*}
Summing up, if $\frac{\lambda_*}{\lambda}\geq \min\{1,
\frac{m-2}4\}$ then Corollary~\ref{cor 3.17} applies, and we
conclude that every non-negative solution of (\ref{0.2'})
vanishes identically. To compare with Theorem~3.9 in \cite{AB}, we
point out we are assuming a less stringent condition on $a_+$ and that
we do not require that $u$ tends to zero at infinity.

On the other hand, assume that $a_+(x)$ satisfies the more
stringent condition assumed in \cite{AB},
\begin{equation*}
a_+(x) \leq \min\{ \frac{k}{|x|^2} , \frac A{|x|^{2+\delta}}\}
\end{equation*}
for some positive constants $A$,  $k$ and $\delta$. According to
\cite{AB}, Theorem~3.5, every positive solution of (\ref{0.2'})
which tends to zero at infinity satisfies the estimate
\begin{equation}
\label{u est}
u(x)\leq \frac{C}{|x|^{m-2}},
\end{equation}
and, in fact, by Theorem~3.4 therein, every positive solution
tends to zero at infinity, provided $a(x)$ is strictly negative
off a compact.
Now, it is easy to see that if $u$ satisfies (\ref{u est}), then
\begin{equation*}
\int_{B_r} u^2 \leq Cr^2,
\end{equation*}
and, clearly, $\lambda_1^{\Delta +\lambda a(x)}(\R^m)\geq 0$
for every $\lambda\leq \lambda_*.$ An application of
Corollary~\ref{cor 3.17} with $\beta +1= H=1$
shows that $u$ vanishes identically. We therefore recover the conclusion
of Theorem~3.9 in \cite{AB}.
}
\end{remark}

Theorem~\ref{thm 3.2} does not cover the "endpoint" case where
$K=-1$ in (\ref{3.4}), which we are going to consider presently.
We therefore assume that there exists a positive solution $\vp$ of
\begin{equation}
\label{3.4bis}
\Delta \vp + H a(x) \vp \leq \frac {|\nabla \vp|^2}\vp, \qquad H>0.
\end{equation}
If $u$ is a $C^2$ solution of (\ref{3.5}) with $\sigma\geq 0$,
we define $v=\vp^{-\gamma} u$, $\gamma\geq 0$.
A computation that uses (\ref{3.4bis}),
(\ref{3.5}) and Young inequality yields
\begin{equation*}
\begin{split}
v\Delta v & =(\gamma H -1) a(x) \vp^{-2\gamma} u^2 + b(x)
\vp^{-2\gamma}u^{\sigma +1} \\
& + \gamma^2\vp^{-2\gamma-2} u^2 |\nabla \vp|^2
-A \vp^{-2\gamma} |\nabla u|^2 -2\gamma \vp^{-2\gamma-1} u\langle
\nabla u,\nabla \vp\rangle \\
& \geq (\gamma H -1) a(x) \vp^{-2\gamma} u^2 + b(x)
\vp^{-2\gamma}u^{\sigma +1} \\
& +  \gamma^2(1-\frac 1 \epsilon)
\vp^{-2\gamma-2} u^2 |\nabla \vp|^2
-(A +\epsilon)\vp^{-2\gamma} |\nabla u|^2
\end{split}
\end{equation*}
Choosing  $\gamma = 1/H$ and  $\epsilon = -A$, the right
hand side reduces to
\begin{equation*}
b(x) \vp^{-2/H}u^{\sigma +1}
 + \frac 1 {H^2} (1+\frac 1 A)
\vp^{-2/H-2} u^2 |\nabla \vp|^2,
\end{equation*}
and we easily deduce that, if $A\leq -1$,
then the function $v$ satisfies
\begin{equation}
\label{3bis}
\Delta v\geq b(x)
\vp^{(\sigma-1)/H}v^{\sigma}.
\end{equation}

Using (\ref{3bis}) we obtain the following version of
Theorem~\ref{thm 3.2}.

\begin{theorem}
\label{thm 3.2'}
Let $a(x),$ $b(x)\in C^0(M)$, with $b(x)\geq 0,$ and assume
that $\vp$ is a positive $C^2$ solution of (\ref{3.4bis})
satisfying
\begin{equation}
\label{5bis}
\vp(x)\geq C r(x)^{1/\delta}
\end{equation}
for $ r(x)\gg 1$, and some constants $C>0$ and $\delta>0$.
Then the differential inequality
\begin{equation*}
\label{3.5bis}
u\Delta u + a(x)u^2 - b(x) u^{\sigma +1} \geq -A |\nabla u|^2,
\end{equation*}
with $\sigma \geq 0$ and $A\leq -1$, has no non-negative $C^2$
solution satisfying
\begin{equation}
\label{6bis}
\mathrm{supp }\,u \cap \{x:M\,:\, b(x)>0\}\neq \emptyset
\end{equation}
and
\begin{equation}
\label{7bis}
\frac{r^{\delta p}}{\int_{\bdr B_r} u^p}\not\in L^1(+\ty)
\end{equation}
for some $p>1.$
\end{theorem}

\begin{proof}
According to (\ref{3bis}) above, the function $v=\vp^{-1/H}u$ is
sub-harmonic. Further, (\ref{5bis}) and (\ref{7bis}) imply that
\begin{equation*}
\Bigl(\int_{\bdr B_r} v^p \Bigr)^{-1}\not \in L^1(++\ty).
\end{equation*}
An application of Theorem~B in \cite{RS} shows  that $v$ is
constant. The conclusion now follows as in the proof of
Theorem~\ref{thm 3.2}.
\end{proof}

Note that, even in the case of Theorem~\ref{thm 3.2'}, if $A<-1,$
then the conclusion can be strengthened to assert that every
non-negative solution of (\ref{3.5bis}) vanishes identically, unless
$a(x)=b(x)\equiv 0.$

In applying Theorem~\ref{thm 3.2'} it is of course crucial to
being able to find positive solutions of (\ref{3.4bis}) satisfying
the asymptotic lower bound (\ref{5bis}).
By contrast, in order to apply Theorem~\ref{thm 3.2} one needs a positive
solution of (\ref{3.4}), whose existence, in typical applications
like the one exemplified by Theorem~\ref{thm 3.3} above, is guaranteed
by means of assumptions on the spectrum of a suitable operator.

Observe now that if $v$ is a solution of the Poisson equation
\begin{equation*}
\Delta v = a(x)
\end{equation*}
then the function $\vp = e^{-v}$ is a positive solution of
\begin{equation*}
\Delta \vp + a(x) \vp = \frac{|\nabla\vp|^2}\vp,
\end{equation*}
and furthermore, an upper bound for $v$ yields a lower bound for $\vp.$

The Poisson equation
on complete Riemannian manifolds has been extensively studied using heat
kernel techniques to obtain bounds on the Green kernel.  To illustrate
an application of Theorem~\ref{thm 3.2'}, we consider the  elementary case
where the positive part of $a(x)$ is integrable.  Then we have the following
lemma (see, e.g., the proof of Theorem 3.2 in \cite{NST})

\begin{lemma}
\label{lemma Poisson}
Let $(M, \langle \,,\,\rangle)$ be a complete, non-parabolic
manifold, and let $\rho\in C^{0,\alpha}(M)\cap L^1(M)$
($0\leq \alpha< 1 $) be a non-negative
function.
Then, there exists a solution $v\in C^{2}$ of the Poisson equation
\begin{equation*}
\Delta v = \rho
\end{equation*}
satisfying $v\leq 0.$
\end{lemma}

\begin{proof} Let $G(x,y)$ be the
Green kernel, i.e., the minimal positive fundamental solution of
the Laplacian, which exists by  the assumption that $M$ is non-parabolic.
The Green kernel is symmetric and, if $\psi\in C_c^{\ty}(M)$,
then the function $u(x) = -\int_M G(x,y) \psi(y)$ is smooth and satisfies
$\Delta u=\psi.$

We claim that if  $\rho\in C^0(M)\cap L^1(M)$ then the function
$v = -\int_M G(x,y) \rho(y)$ is well defined and locally bounded.
Assuming the claim, for every  $\psi\in C_c^\ty(M)$ we have
\begin{multline*}
\int _M v\Delta \psi = - \int_M \psi(x) \int_M \Delta  G(x,y) \rho(y)
\\
= - \int_M\rho(y) \int_M G(x,y) \Delta \psi(x) = \int_M \rho(y)
\psi(y),
\end{multline*}
so that $v$ satisfies the Poisson equation in distributional
sense, and therefore, by standard elliptic regularity
(see \cite{A}), Theorem~3.55), it is a classical solution.
Clearly, $v$ is non-positive.

To prove the claim, fix $R>0$ and for every $x\in B_R$ we write
\begin{equation}
\label{Poisson2}
\int_M G(x,y) \rho (y) =
\int_{B_{2R}} G(x,y) \rho (y) +\int_{M\setminus B_{2R}}  G(x,y) \rho (y)
\end{equation}
Since $G(x,y)$ is  locally integrable uniformly for $x\in B_{R}$, the first
integral on the right hand side is bounded above by a constant
independent of $x\in B_R$.

On the other hand, by the local Harnack
inequality there exists a constant $C$ independent of $x\in B_R$
and  such that
\begin{equation*}
G(x,y)\leq C G(o,y) \qquad \text{for every }\,
y\in M\setminus B_{2R}.
\end{equation*}
Moreover,
\begin{equation*}
\sup_{M\setminus B_{2R}} G(o,y)<+\ty.
\end{equation*}
Indeed, let $\Omega_n$ be an exhaustion of $M$ by open sets containing $o$
and with  smooth boundary and let $G_n$ by the Green kernel of
$\Omega_n$ and recall that,  by the standard construction  of the Green
kernel $G(x,y)$, $G_n(x,y) \to G(x,y)$ locally uniformly in
$M\setminus\{x\} $. Let $C>\sup_{\bdr B_{2R}} G(o,y)$, then, for
every sufficiently large $n$ we have $C> G(o,y)\geq G_n(o,y)$
for $y\in \bdr B_{2R}$ and clearly $C> G_n(o,y)=0$ for $y\in \bdr
\Omega_n$. Thus, by the comparison principle, $C>G_n(o,y)$ in
$\Omega\setminus B_{2R}$, whence, letting $n\to +\ty,$ $G(o,y)\leq C$
for $y\in M\setminus B_{2R}.$
It follows that there exists a constant $C'$ independent of $x\in
B_{R}$ and $y\in M\setminus B_{2R}$ such that
\begin{equation*}
G(x,y)\leq C'.
\end{equation*}
Since $\rho$ is integrable, this implies that the second integral on
the right hand side of (\ref{Poisson2}) is also bounded independently
of $x\in B_R$, as required to complete the proof of the claim.
\end{proof}

\begin{corollary}
\label{cor 3.2''}
Let $(M, \langle \,,\,\rangle)$ be a complete, non-parabolic
manifold,  let the functions  $a(x) \in C^{0,\alpha}(M)$, and $b(x)\in C^0(M)$ satisfy
$b(x)>0$ and
\begin{equation}
\label{2ter}
a_+(x) \in L^1(M),
\end{equation}
and suppose that for some constants   $\sigma>1,$
$A\in (-\ty , -1]$ and $\mu, p, q$ satisfying
\begin{equation*}
q> \max \{1, 3-\sigma\} \quad 0\leq \mu \leq
2\frac{\sigma -1}{\sigma +q-2},\quad p>\frac{q+\sigma-2}{\sigma-1},
\end{equation*}
we have
\begin{align}
\label{3ter}
&\int_{B_r} a_+(x)^p = O\Bigl( r^{[2(\sigma -1) - \mu(q+\sigma -2)]
\frac{p -1}{q-1}}\Bigr) \quad \text{as }\,\, r\to +\ty\\
\label{4ter}
&\vol B_r= O\Bigr( r^{2+ (2-\mu)\frac{\sigma +q -2}{\sigma-1}}
\Bigr) \quad \text{as }\,\, r\to +\ty\\
\label{5ter}
&b(x) \geq \frac{C}{r(x)^{\mu}} \quad \text{for }\,\, r(x)\gg 1.
\end{align}
Then there are no non-negative, non-identically zero $C^2(M)$ solutions
of the differential inequality
\begin{equation}
\label{6ter}
u\Delta u +a(x) u^2 \geq b(x) u^{\sigma +1} - A |\nabla u|^2 \quad
\text{on }\, M.
\end{equation}
\end{corollary}

\begin{proof}
Since $a_+$ is integrable, by Lemma~\ref{lemma Poisson} and
the  preceding discussion  there exists a solution $\vp\geq 1$
of
\begin{equation*}
\Delta \vp + a_+ \vp = \frac{|\nabla \vp |^2}{\vp},
\end{equation*}
and $\vp$ is a solution of the differential inequality
(\ref{3.4bis}).

Now, let $u$ be a non-negative solution of (\ref{6ter}).
Noting that  $q>1,$ and $A\leq -1$ imply $q>A+2$, applying
Lemma~\ref{lemma 3.1} and using the lower bound for $b(x)$
(\ref{5ter})
imply that
\begin{equation}
\label{int upp bound}
\int_{B_r}  u^{q+\sigma -2} \leq
{C_1}{r^{(\mu-2) \frac{q+\sigma-2}{\sigma-1}}} \vol {B_{2r}}
 +
C_2
{r^{\mu\frac{q+\sigma-2}{\sigma-1}}}
\int_{B_{2r}}
a_+(x)^{\frac{\sigma +q-2}{\sigma-1}}.
\end{equation}
We claim that  (\ref{3ter}) implies
\begin{equation}
\label{7ter}
\int_{B_{2r}}
a_+(x)^{\frac{\sigma +q-1}{\sigma-1}}=
O\Bigl( r^{2-\mu\frac{\sigma+q-2}{\sigma -1}}\Bigr)
\qquad \text{as }\, \,
r\to +\ty,
\end{equation}
which, together with the  volume growth  assumption (\ref{4ter})
yields
\begin{equation*}
\int_{B_r}  u^{q+\sigma -2} = O\bigl(r^2\bigr)
\qquad \text{as }\, \,
r\to +\ty.
\end{equation*}
As in the proof of Theorem~\ref{thm 3.3}, it follows
(see, e.g., \cite{RS} Proposition~1.3)  that
$u$ satisfies condition (\ref{7bis}) with  $\delta=0$  and
exponent  $q+\sigma-2$ which is greater than $1$ by the conditions
on $q$. Thus, Theorem~\ref{thm 3.2'} (with $a(x)$ replaced  by $a_+(x)$)
applies and $u$  vanishes  identically.

To conclude it remains to prove the claim. To this end,
we set $p'=(q+\sigma-2)/(\sigma -1)$, and apply H\"older
inequality with conjugate exponents  $(p-1)/(p-p')$ and
$(p-1)/(p'-1)$  to estimate
\begin{equation*}
\begin{split}
\int_{B_r} a_+ (x)^{p'}
&= \int_{B_r} a_+ (x)^{\frac{p-p'}{p-1}+p\frac{p'-1}{p-1}}\\
 &\leq
 \Bigl(\int_{B_{2r}} a_+(x)\Bigr)^{\frac{p-p'}{p-1}}
 \Bigl(\int_{B_{2r}} a_+(x)^p\Bigr)^{\frac{p'-1}{p-1}}\\
&= O\Bigl( r^{[2(\sigma-1)-\mu(q+\sigma-2)]\frac{p'-1}{q-1}}\Bigr)
=O\Bigl( r^{2-\mu \frac{q+\sigma-2}{\sigma-1}}\Bigr),
\end{split}
\end{equation*}
as required.
\end{proof}

\begin{remark}
\label{rmk 8ter}
{\rm
Assume that  $b(x)$ satisfies the condition stated in the corollary, with
$\mu<\sigma -1$, and that conditions (\ref{3ter}) and (\ref{4ter}) are
replaced by
\begin{align*}
&\int_{B_r} a_+(x)^{\frac{2}{\sigma -1}} = O\Bigl(
r^{
2[1-\frac{\mu }{\sigma -1}]
}
\Bigr) \quad \text{as }\,\, r\to +\ty\\
&\vol B_r= O\Bigr(
r^{
2[1-\frac{\mu - 2}{\sigma -1}]
}
\Bigr) \quad \text{as }\,\, r\to +\ty.
\end{align*}
It follows from  (\ref{int upp bound}) above with $q+\sigma -2 = 2$,
that every non-negative solution of
\begin{equation}
\label{6ter''}
\Delta u + a(x) u - b(x) u^\sigma = 0
\end{equation}
satisfies
\begin{equation}
\label{integr cond}
\int_{B_r} u^2 \leq C r^2\log r,
\end{equation}
and the same estimate is clearly satisfied by the difference of
two solutions. An application of Theorem~4.1 in \cite{BRS2} shows
that (\ref{6ter''}) has at most one positive solution. We remark in
this respect that if we replace $u^2$ in
(\ref{integr cond}) with $u^p$ with $p> 2$, then the conclusion of Theorem~4.1 in
\cite{BRS2} fails, as the example described on pages 214-215 therein
shows.
}
\end{remark}

\end{document}